\documentclass[12pt,reqno]{amsart}

\usepackage[T1]{fontenc}
\usepackage[utf8]{inputenc}
\usepackage{amsmath,amsfonts,amsthm,amssymb,amsxtra,bbm}
\usepackage{fourier,setspace,graphicx,color,pdflscape}
\usepackage[colorlinks=true]{hyperref}
\hypersetup{urlcolor=blue, linkcolor=blue, citecolor=red, anchorcolor=blue]}
\usepackage{graphics}
\usepackage{epsfig}
\usepackage{graphicx}
\usepackage{epstopdf}
\usepackage{floatrow}
\usepackage{setspace}

\newtheorem{theorem}{Theorem}
\newtheorem{proposition}[theorem]{Proposition}
\newtheorem{lemma}[theorem]{Lemma}
\newtheorem{corollary}[theorem]{Corollary}
\newtheorem{definition}[theorem]{Definition}
\newtheorem{remark}[theorem]{Remark}

\newcommand{\R}{{\mathbb R}}
\newcommand{\N}{{\mathbb N}}
\newcommand{\Z}{{\mathbb Z}}
\newcommand{\C}{{\mathbb C}}
\newcommand{\Sp}{{\mathbb S^2}}
\newcommand{\be}[1]{\begin{equation}\label{#1}}
\newcommand{\ee}{\end{equation}}
\renewcommand{\(}{\left(}
\renewcommand{\)}{\right)}
\newcommand{\ird}[1]{\int_{\R^3}{#1}\,dx}
\newcommand{\ir}[1]{\int_{\R}{#1}\,ds}
\newcommand{\is}[1]{\int_{\Sp}{#1}\,d\omega}
\newcommand{\ic}[1]{\iint_{\R\times\Sp}{#1}\,ds\,d\omega}

\newcommand{\nrmr}[2]{\left\|{#1}\right\|_{\mathrm L^{#2}(\R)}}
\newcommand{\nrms}[2]{\left\|{#1}\right\|_{\mathrm L^{#2}(\Sp,d\omega)}}
\newcommand{\nrmc}[2]{\|{#1}\|_{\mathrm L^{#2}(\R\times\Sp)}}
\renewcommand{\Re}{\mathrm{Re}}
\renewcommand{\Im}{\mathrm{Im}}
\newcommand{\lrangle}[2]{\left\langle{#1},{#2}\right\rangle}
\newenvironment{psmallmatrix}{\(\begin{smallmatrix}}{\end{smallmatrix}\)}

\usepackage{etoolbox}
\newcounter{taggedeq}
\pretocmd{\equation}{\stepcounter{taggedeq}}{}{}

\makeatletter
\@namedef{subjclassname@2020}{%
\textup{2020} Mathematics Subject Classification}
\makeatother
\newcommand{\msc}[1]{\href{https://zbmath.org/classification/?q=cc:#1}{#1}}

\definecolor{darkgreen}{rgb}{.0,.4,.2}

\newcommand{\greenbw}{green} \newcommand{\redbw}{red} 

\title[The CKN inequality for spinors: symmetry and symmetry breaking]{The CKN inequality for spinors:\\ symmetry and symmetry breaking}

\author[J.~Dolbeault]{Jean Dolbeault}
\address{\hspace*{-12pt}J.~Dolbeault: CEREMADE (CNRS UMR n$^\circ$ 7534), PSL university, Universit\'e Paris-Dauphine,\newline Place de Lattre de Tassigny, 75775 Paris 16, France}
\email{dolbeaul@ceremade.dauphine.fr}

\author[M.J.~Esteban]{Maria J.~Esteban}
\address{\hspace*{-12pt}M.J.~Esteban: CEREMADE (CNRS UMR n$^\circ$ 7534), PSL university, Universit\'e Paris-Dauphine,\newline Place de Lattre de Tassigny, 75775 Paris 16, France}
\email{esteban@ceremade.dauphine.fr}

\author[R.L.~Frank]{Rupert L.~Frank}
\address{\hspace*{-12pt}R.L.~Frank: Department of Mathematics, LMU Munich, Theresienstr. 39, 80333 M\"unchen,\newline Germany; Munich Center for Quantum Science and Technology, Schellingstr.~4,\newline 80799 M\"unchen, Germany; Mathematics 253-37, Caltech, Pasadena, CA 91125, USA}
\email{r.frank@lmu.de}

\author[M.~Loss]{Michael Loss}
\address{\hspace*{-12pt}M.~Loss: School of Mathematics, Georgia Institute of Technology Atlanta, GA 30332,\newline United States of America}
\email{loss@math.gatech.edu}

\begin{document}
\begin{abstract} This paper is devoted to Sobolev interpolation inequalities for spinors, with weights of Caffarelli-Kohn-Nirenberg (CKN) type. In view of the corresponding results for scalar functions, a natural question is to determine whether optimal spinors have symmetry properties, or whether spinors with symmetry properties are linearly unstable, in which case we shall say that symmetry breaking occurs. What symmetry means has to be carefully defined and the overall picture turns out to be richer than in the scalar case. So far, no symmetrization technique is available in the spinorial case. We can however determine a range of the parameters for which symmetry holds using a detailed analysis based mostly on spectral methods.
\\[4pt]
\noindent{\sc Résumé.} Cet article est consacré à des inégalités d'interpolation de Sobolev pour les spineurs, avec des poids de type Caffarelli-Kohn-Nirenberg (CKN). Au vu des résultats correspondants pour les fonctions scalaires, une question naturelle consiste à déterminer si les spinors optimaux possèdent des propriétés de symétrie, ou si les spinors dotés de propriétés de symétrie sont linéairement instables, auquel cas on dira qu'une brisure de symétrie se produit. La notion de symétrie doit être définie avec précision, et le tableau d'ensemble s'avère plus riche que dans le cas scalaire. À ce jour, aucune technique de symétrisation n'est disponible pour les spineurs. Nous pouvons néanmoins déterminer un intervalle de paramètres pour lequel la symétrie est vérifiée grâce à une analyse détaillée, basée principalement sur des méthodes spectrales.
\end{abstract}

\date{July 13, 2026. {\em File:} \textsf{\jobname.tex}}
\maketitle
\vspace*{-0.6cm}
\tableofcontents
\thispagestyle{empty}

\section{Introduction and main result (I)}\label{Sec:Introduction}

A perennial question in the calculus of variation concerns the symmetry properties of optimizers.
A functional $\mathcal F$ is invariant under a group $G$ if $ \mathcal F(g^*\psi)=\mathcal F(\psi)$ for all~$\psi$
and for all $g \in G$, where $g^*\psi$ denotes the action of $g$ on $\psi$. If $\psi_{\!m}$ is an optimizer for the functional $\mathcal F$,
is it true that $\psi_{\!m}$ is invariant under the action of the group $G$?
This is often not the case; the symmetry is then broken.
Exceptions to this rule are, in general, difficult to find and can pose considerable mathematical challenges.
An example where these features are nicely displayed are the Caffarelli-Kohn-Nirenberg inequalities~\cite{MR768824} (also see~\cite{MR0136692} for an earlier related result),
\begin{equation}\label{CKN}\tag{CKN}
\int_{\R^d} \frac{|\nabla v(x)|^2}{|x|^{2a}}\,dx \ge C_{a,b,d}\left(\int_{\R^d} \frac{|v(x)|^p}{|x|^{bp}}\,dx \right)^{2/p}\,.
\end{equation}
Here $d\ge 3$ is an integer, $a$, $b$ are real numbers with $a \le b \le a+1$ and, to ensure scale invariance, $p = \frac{2d}{d-2-2a+2b}$. This inequality is often considered only for real functions, but extends trivially to complex-valued functions (see Remark~\ref{Rem:ComplexCKN}).
The problem is to find the optimal constant $C_{a,b,d}$, which is positive if $a\neq(d-2)/2$,
to discuss whether $C_{a,b,d}$ is achieved or not, and in case optimizers exist, to analyse their main qualitative properties, like positivity and symmetry.
If one replaces the function $v(x)$ by $v\big(R^{-1}x\big)$, where $R$ is any rotation, then the value of the functional does not change.
Thus, the expectation is that the optimal function, when it exists, is radial.
It was shown by Catrina and Wang~\cite{MR1794994}, with later improvements by Felli and Schneider~\cite{MR1973285}, that
below a certain curve $C_{\rm FS}^{(d)}(a,b)=0$ in the parameter space $\{a,b\}$, the rotational symmetry of the optimizer is broken.
In a later work~\cite{MR3570296}, the rotational symmetry in the remaining region was established.

A natural continuation of this line of research is the analogous question for \emph{spinor valued functions}. While this question could be studied in any dimension $d$, below we restrict ourselves to the case where $d=3$. The scalar variable $v$ in~\eqref{CKN} is replaced by the $\C^2$-valued function, or `2-spinor',
$$
\psi(x) = \left(\begin{array}{c} \psi_1(x) \\\psi_2(x) \end{array}\right),
$$
and the gradient is replaced by
$$
\sigma\cdot\nabla\psi = \sum_{j=1}^3 \sigma_{\!j}\,\partial_{\!j}\psi \,.
$$
Here, $\sigma=(\sigma_{\!j})_{j=1,2,3}$ is the family of the Pauli matrices
$$
\sigma_1=\left(\begin{array}{cc}0 & 1 \\ 1 & 0 \end{array}\right)\,,\quad\sigma_2= \left(\begin{array}{cc} 0 & -i \\ i & 0 \end{array}\right)\,,\quad\sigma_3= \left(\begin{array}{cc}1 & 0 \\ 0 & -1 \end{array}\right)\,.
$$
The role of the functional inequality~\eqref{CKN} is played by the inequality
\begin{equation}\label{SCKN}\tag{SCKN}
\int_{\R^3} \frac{|\sigma\cdot \nabla \psi (x)|^2}{|x|^{2\,\alpha}}\,dx \ge \mathcal C_{\alpha,\beta}\left(\int_{\R^3} \frac{|\psi(x)|^p}{|x|^{\beta\,p}}\,dx \right)^{2/p}\,.
\end{equation}
For the same scaling reasons as above, the exponent $p$ is given by
\begin{equation}\label{param}
p = \frac{6}{1-2\,\alpha+2\,\beta}\,.
\end{equation}
Here we consider only the dimension $d=3$ and assume that
$$
\alpha\le\beta\le\alpha+1\,,
$$
which means that $p\in[2,6]$ because of~\eqref{param}. By definition, $\mathcal C_{\alpha,\beta}\ge0$ is the best constant in the space
$$
\mathcal D_{\alpha, \beta}:=\Big\{ \psi:\,|x|^{-\beta }\psi\in \mathrm L^p(\R^3, \C^2)\;;\;\; |x|^{-\alpha}\,\sigma\cdot \nabla \psi \in \mathrm L^2(\R^3, \C^2)\Big\}\,.
$$
In contrast to the scalar inequality~\eqref{CKN}, an infinite set of special values of $\alpha$ appears for which the inequality fails. We set
\be{Lambda}
\Lambda :=\left\{k-\tfrac12\,:\,k\in\Z\setminus\{0\}\right\}\,.
\ee
\begin{proposition}\label{validity}
If $\alpha\in\Lambda$, then $\mathcal C_{\alpha,\beta}=0$ for all $\alpha\leq\beta\leq\alpha+1$. If $\alpha\not\in\Lambda$, then $\mathcal C_{\alpha,\beta}>0$ for all $\alpha\leq\beta\leq\alpha+1$.
\end{proposition}
Our main interest in this paper is in the sharp constant $\mathcal C_{\alpha,\beta}$ and the corresponding set of optimizers.
The chief reason for investigating such a problem is that spinors are ubiquitous in many of the equations of physics: for instance, in the Dirac equation, the Pauli equation and in the influence of the L-S coupling on the spectrum of atoms.
In a particular case, Inequality~\eqref{SCKN} can be viewed as a generalization of the Hardy inequality for spinors when $\alpha = -\,1/2$, $\beta=1/2$ and $p=2$,
\begin{equation}\label{Hardy-Dirac}
\int_{\R^3} |x|\,|\sigma\cdot \nabla \psi (x)|^2dx \ge \int_{\R^3} \frac{|\psi(x)|^2}{|x|}\,dx\,,
\end{equation}
where the constant $\mathcal C_{-1/2,1/2}=1$ is sharp. This inequality has proven to be very useful for analyzing Dirac-Coulomb Hamiltonians~\cite{MR1855581, MR1761368, MR3960263, MR4353967, MR4332485, MR4056270}.

The scalar~\eqref{CKN} inequality is invariant under the group O(3) in the sense that both sides of the inequality remain the same if $v$ is replaced by $v_R$ for $R\in$ O(3), where $v_R(x):=v\big(R^{-1}x\big)$. Similarly, we now show that the spinorial inequality~\eqref{SCKN} is invariant under the group SU(2). If $A$ is any SU(2) matrix, then, recalling that $\sigma=(\sigma_k)_{k=1,2,3}$ is a basis for the traceless Hermitean $(2\times 2)$-matrices, one has
\begin{equation}
\label{eq:representation}
A^* \sigma_{\!i} A = \sum_{k=1}^3 R_{ik}(A)\,\sigma_k\,.
\end{equation}
The matrix $R(A) =(R_{ik}(A))$ turns out to be a rotation. The map $A \mapsto R(A)$ is a representation of the group SU(2), i.e., $R(A\,B)=R(A)\,R(B)$.
If one transforms
\[
\psi(x) \mapsto A\,\psi \big(R^{-1}(A)\,x\big) =: \psi_{\!A}(x)\,,
\]
one finds that
$$
\sigma \cdot \nabla \psi_{\!A}(x)
=(\sigma \cdot \nabla \psi)_A(x)\,.
$$
Thus,
$$
\int_{\R^3} \frac{|\sigma\cdot \nabla \psi_{\!A} (x)|^2}{|x|^{2\,\alpha}}\,dx =\int_{\R^3} \frac{ |\sigma\cdot \nabla \psi (x)|^2}{|x|^{2\,\alpha}}\,dx \quad {\rm and} \quad \int_{\R^3} \frac{ |\psi_{\!A}(x)|^p}{|x|^{\beta\,p}}\,dx =\int_{\R^3} \frac{|\psi(x)|^p}{|x|^{\beta\,p}}\,dx \,.
$$
This shows that Inequality~\eqref{SCKN} is SU(2)-invariant.

In the symmetry problem for the scalar~\eqref{CKN} inequality one asks whether some (or any) optimizer is invariant under the O(3) action. In the spinorial case, however, the symmetry question has to be phrased differently, and \emph{this is one of the main differences from the scalar case}. The reason is that there is no nontrivial spinor that is invariant under the SU(2) action. Indeed, if $\psi \in\mathrm L^1_{\rm loc}(\R^3)$ satisfies $\psi_{\!A} = \psi$ for all $A\in$ SU(2), then $\psi\equiv 0$. The proof of this assertion is simple. Indeed, the invariance of the left side of~\eqref{eq:representation} under $A\mapsto -A$ implies that $R(A)=R(-A)$. Thus, if $\psi_{\!A}=\psi=\psi_{-A}$, then $A\,\psi(x) =\psi\big(R(A)\,x\big)=\psi\big(R(-A)\,x\big)=-\,A\,\psi(x)$, so $\psi(x)=0$.

We now explain how to formulate the notion of symmetry in the spinorial case. We consider the uniform probability measure $d\omega$ on $\Sp$, that is, in spherical coordinates, $d\omega= \tfrac1{4\,\pi}\sin\theta\,d\theta\,d\varphi$, and the corresponding Hilbert space $\mathrm L^2( \Sp, \C^2; d\omega)$ of (angular) spinors on the sphere~$\Sp$. For $\xi\in \mathrm L^2( \Sp, \C^2; d\omega)$, let $\xi_{\!A}(\omega):= A\,\xi\big(R^{-1}(A)\,\omega\big)$. Then $\xi\mapsto\xi_{\!A}$ is a representation of SU(2) in $\mathrm L^2( \Sp, \C^2; d\omega)$. Decomposing into irreducible representations we obtain the orthogonal decomposition
$$
\mathrm L^2\( \Sp, \C^2; d\omega\) = \bigoplus_{k\in\Z\setminus\{-1\}} \mathcal H_k\,.
$$
Each subspace $\mathcal H_k$ transforms irreducibly under the transformation $\xi\mapsto\xi_{\!A}$. The integer~$k$ in this decomposition is an eigenvalue of the operator $\sigma\cdot L$, where $L:= \omega\wedge(-i\,\nabla)$ is the angular momentum operator (see~\cite[p.~8 \& pp.~123--127]{MR1219537}). We recall the details of this decomposition in Appendix~\ref{App5}.

An elementary calculation shows that
\begin{equation} \label{commutation}
(\sigma \cdot \omega) (\sigma\cdot L +1) = -\,(\sigma\cdot L +1)\,(\sigma \cdot \omega)
\end{equation}
for all $\omega\in\Sp$. Since $\sigma\cdot\omega$ provides a unitary equivalence between $\mathcal H_k$ and $\mathcal H_{-2-k}$, we see that the representations for $k \ge 0$ and $k \le -\,2$ are equivalent.

The irreducible subspaces of lowest dimension are $\mathcal H_0$ and $\mathcal H_{-2}$ and the dimension of each of these two spaces is $2$. Thus, we propose to reformulate the notion of spherical symmetry of an optimizer of~\eqref{SCKN} as {\it belonging to an irreducible representation of lowest dimension}. This should be compared with the notion of symmetric functions in the setting of the scalar inequality~\eqref{CKN}, where the irreducible subspace of the lowest dimension of the representation of O(3) has dimension one and corresponds to spherically symmetric functions.

Using the fact that $\mathcal H_0$ consists of constant spinors, and consequently $\mathcal H_{-2}$ consists of $\sigma\cdot\omega$ times constant spinors, we arrive at the following definition.
\begin{definition} A spinor $\psi$ on $\R^3$ is \emph{symmetric} if there is a scalar, complex-valued function $f$ on $\R_+$ and a constant spinor $\chi_0\in\C^2$ such that either
$$
\psi(x) = f(|x|) \, \chi_0
\quad\text{for a.e.}\ x\in\R^3 \,,
$$
or
$$
\psi(x) = f(|x|) \, \sigma\cdot\tfrac{x}{|x|}\, \chi_0
\quad\text{for a.e.}\ x\in\R^3 \,.
$$
We call these symmetric spinors of type $\mathcal H_0$ and $\mathcal H_{-2}$, respectively.
\end{definition}
Let us define $\mathcal C^\star_{\alpha,\beta}$ as the best constant in the Inequality~\eqref{SCKN} when restricted to \emph{symmetric spinors}. This constant can be easily computed and the optimizers can be classified explicitly, as will be shown in Proposition~\ref{symmetricproblem}. In particular, symmetric optimizers are of type $\mathcal H_0$ if $\alpha> -1/2$ and of type $\mathcal H_{-2}$ if $\alpha< -1/2$.

\medskip These preliminary considerations raise an obvious question:\\
\centerline{\emph{For which $\alpha$ and $\beta$ are the global optimizers for Inequality~\eqref{SCKN} symmetric ?}}
\\Symmetry regions are shown in green 
Before stating our main result concerning this question, we mention one more special feature of the family of Inequalities~\eqref{SCKN}. Namely, given $\psi\in \mathcal D_{\alpha,\beta}$, let us set
$$
\widetilde\psi(x) := |x|^{-2\,\alpha-1}\,\sigma\cdot\tfrac{x}{|x|}\, \psi(x) \,.
$$
Then one has $\widetilde\psi \in \mathcal D_{-\alpha-1,\beta-2\,\alpha-1}$ and
$$
\int_{\R^3} \frac{|\sigma\cdot \nabla \widetilde\psi (x)|^2}{|x|^{-2\,(\alpha+1)}}\,dx
= \int_{\R^3} \frac{|\sigma\cdot \nabla \psi (x)|^2}{|x|^{2\,\alpha}}\,dx
\quad\text{and}\quad
\int_{\R^3} \frac{|\widetilde\psi(x)|^p}{|x|^{(\beta-2\,\alpha-1) p}}\,dx
= \int_{\R^3} \frac{|\psi(x)|^p}{|x|^{\beta\,p}}\,dx \,.
$$
These identities follow by a straightforward computation, but they will become obvious later in this introduction. As a consequence, we deduce that
$$
\mathcal C_{\alpha,\beta} = \mathcal C_{-\alpha-1,\beta-2\,\alpha-1}
\quad\text{and}\quad
\mathcal C^\star_{\alpha,\beta} = \mathcal C^\star_{-\alpha-1,\beta-2\,\alpha-1} \,.
$$
Moreover, minimizers for the inequality with parameters $(\alpha,\beta)$ are in one-to-one correspondence with minimizers for the inequality with parameters $(-\alpha-1,\beta-2\,\alpha-1)$. Note that the exponent $p$ in~\eqref{param} is the same for both parameter pairs. Moreover, parameter pairs with $\alpha=-\,\frac12$ are invariant under the transformation. As a consequence of this transformation we restrict ourselves in the following theorem to parameter values $\alpha\geq -1/2$.

The following is a first version of our main result.
\begin{theorem}\label{Cor:main}
There are optimizers for~\eqref{SCKN} for all $\alpha\in\R\setminus\Lambda$ and $\alpha<\beta<\alpha+1$. Moreover, there are two functions $\overline\alpha_\star$ and $\overline\alpha^\star$ on $(0,1)$, taking values respectively in $(-1/2,0]$ and $[1/2,1)$, such that the following holds for all $\alpha\in[-1/2,\infty)\setminus\Lambda$ and $\alpha<\beta<\alpha+1$:
\begin{enumerate}
\item[\rm(i)] if either $\overline\alpha_\star(\beta-\alpha)<\alpha<1/2$ or $1/2<\alpha<\overline\alpha^\star(\beta-\alpha)$, then $\mathcal C_{\alpha,\beta}=\mathcal C^\star_{\alpha,\beta}$ and all optimizers are then symmetric,
\item[\rm(ii)] if either $-\,1/2\le\alpha<\overline\alpha_\star(\beta-\alpha)$ or $\alpha>\overline\alpha^\star(\beta-\alpha)$, then $\mathcal C_{\alpha,\beta}<\mathcal C^\star_{\alpha,\beta}$
\end{enumerate}
The corresponding results are valid for $\alpha\in(-\infty,-1/2]$ in view of the symmetry of the problem under the map $(\alpha,\beta)\mapsto(-\alpha-1,\beta-2\,\alpha-1)$.
\end{theorem}
Theorem~\ref{Cor:main} is only a qualitative version of what we will prove. Namely, we will give quantitative upper and lower bounds on the functions $\overline\alpha_\star$ and $\overline\alpha^\star$; see Remark~\ref{translation}.
A graphical representation of what we prove can be seen in Figure~\ref{Fig:SCKN}.

Already on a qualitative level our results concerning the spinorial inequality~\eqref{SCKN} are quite different from the results concerning the scalar inequality~\eqref{CKN}. In the spinorial case, each line $\beta-\alpha\equiv\mathrm{const}\in(0,1)$ intersects the symmetry region in two bounded intervals, while such an intersection by a line $b-a\equiv\mathrm{const}\in(0,1)$ occurs only in a single interval in the scalar case (up to the case $a=1/2$ which is trivial because $C_{1/2,b,3}=0$). Another difference is that in the spinorial case symmetry can occur only for $-\,2\leq\alpha\leq1$, while in the scalar case it can occur for arbitrarily large $a\in\R$. We also note that in the spinorial case the figure is symmetric with respect to $\alpha=-\,1/2$, while it is symmetric with respect to $a=+1/2$ in the scalar case. Our results are summarized in Figure~\ref{Fig:SCKN}.\clearpage

\begin{figure}[ht]
\includegraphics[width=12cm]{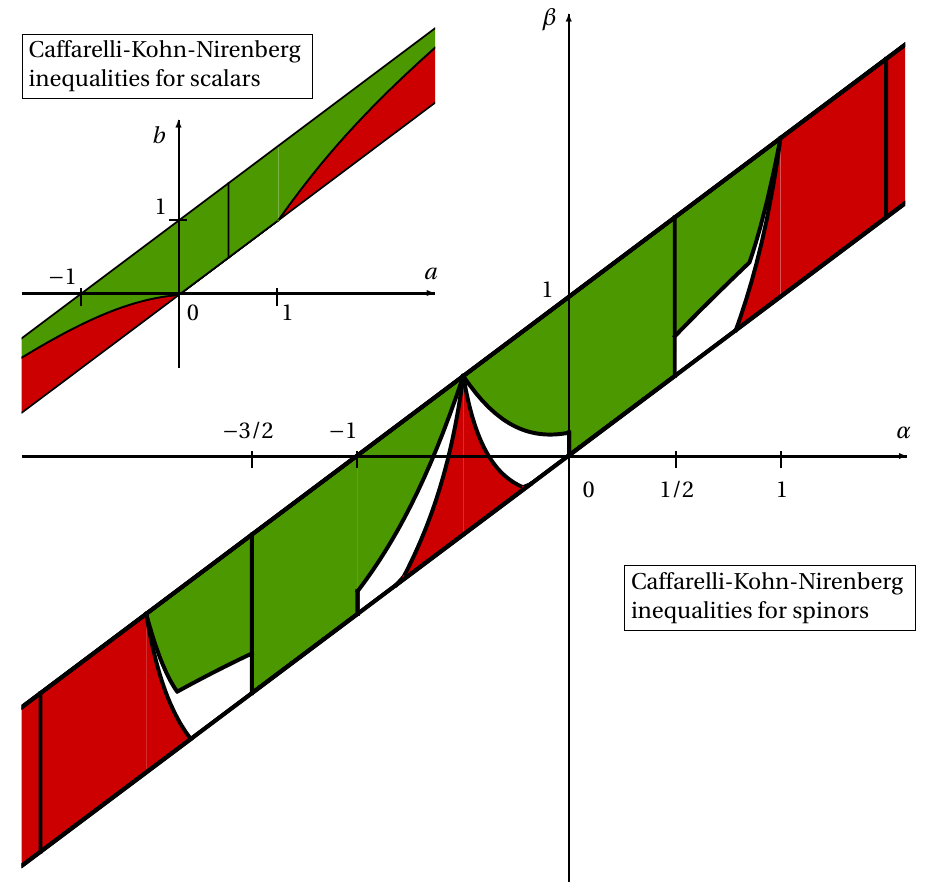}
\caption{\small\label{Fig:SCKN} Symmetry regions are shown in \greenbw\ and symmetry breaking regions appear in \redbw. The upper left corner represents the parameter domain for~\eqref{CKN} Inequalities (scalar functions), while the main figure is the parameter domain for~\eqref{SCKN} Inequalities (spinors). In the latter case, the threshold between symmetry and symmetry breaking is determined by the functions $\overline\alpha_\star$ and $\overline\alpha^\star$ which take values in the white areas of the strip $\alpha\le\beta\le\alpha+1$.}
\end{figure}

\section{Passing to logarithmic variables and main result (II)}

Our goal in this section is to bring Inequality~\eqref{SCKN} into an equivalent form and state a more precise version of Theorem~\ref{Cor:main} in this equivalent reformulation.

A crucial feature of Inequality~\eqref{SCKN}, just like of~\eqref{CKN}, is its scale invariance. In logarithmic coordinates, this invariance becomes translation invariance, and these coordinates have been useful in the study of Inequality~\eqref{CKN}. Correspondingly, we will now state an equivalent form of Inequality~\eqref{SCKN} in logarithmic coordinates.
\begin{lemma}\label{Lem:chvar} If we set $s =\log r$, $r = |x|\neq0$, $\omega = x/r$ and define a function $\phi$ on $\R\times\Sp$ by $\psi(x) = r^{\alpha -1/2}\,\phi(s,\omega)$, then Inequality~\eqref{SCKN} becomes
\be{logarithmicSCKN}\tag{SCKN$_{\rm log}$}
\ic{\(|\partial_s\phi|^2+\big|\big(\sigma\cdot L-\alpha+\tfrac12\big)\,\phi\big|^2\)}\ge C_{\alpha,p}\(\ic{|\phi|^p}\)^{2/p}.
\ee
The optimal constant is $C_{\alpha,p}=(4\,\pi)^{(2-p)/p}\,\mathcal C_{\alpha,\beta}$, where $\alpha$, $\beta$ and $p$ are related by~\eqref{param}.
\end{lemma}
We parametrize the constant in~\eqref{logarithmicSCKN} by $\alpha$ and $p$, because $\beta$ does not appear explicitly in this inequality. The parameters under consideration are now $\alpha\in\R$ and $p\in[2,6]$. According to Proposition~\ref{validity} we have $C_{\alpha,p}>0$ if and only if $\alpha\in\R\setminus\Lambda$.

In logarithmic coordinates the before-mentioned symmetry between $(\alpha,\beta)$ and $(-\alpha-1,\beta-2\,\alpha-1)$ can be seen as follows. For an arbitrary function $\phi$, using~\eqref{commutation} and setting $\eta=(\sigma\cdot\omega)\,\phi$,
\eqref{logarithmicSCKN} takes the form
\begin{equation} \label{logarithmicalt}
\ic{\(|\partial_s\eta|^2+\big|\big(\sigma\cdot L+\alpha+\tfrac32\big)\,\eta\big|^2\)}\ge C_{\alpha,p}\(\ic{|\eta|^p}\)^{2/p}.
\end{equation}
This puts into evidence that
\[
C_{-\,(\alpha+1),p}=C_{\alpha,p}\,.
\]
An optimizer for~\eqref{logarithmicSCKN} is transformed by $\phi\mapsto(\sigma\cdot\omega)\,\phi$ into an optimizer for~\eqref{logarithmicalt} and reciprocally because $(\sigma\cdot\omega)^2=1$.

The concept of symmetry translates as follows to logarithmic coordinates.
\begin{definition}\label{symmetric} A spinor $\phi$ on $\R\times\Sp$ is \emph{symmetric} if there is a scalar, complex-valued function $u$ on $\R$ and a constant spinor $\chi_0\in\C^2$ such that either
$$
\phi(s,\omega) = u(s) \, \chi_0
\quad\text{for a.e.}\ (s,\omega)\in\R\times\Sp \,,
$$
or
$$
\phi(s,\omega) = u(s) \, \sigma\cdot\omega \, \chi_0
\quad\text{for a.e.}\ (s,\omega)\in\R\times\Sp \,.
$$
We call these symmetric spinors of type $\mathcal H_0$ and $\mathcal H_{-2}$, respectively.
\end{definition}
Let $C^\star_{\alpha,p} = (4\,\pi)^{(2-p)/p}\,\mathcal C^\star_{\alpha,\beta}$ be the best constant in~\eqref{logarithmicSCKN} restricted to the space of symmetric functions $\phi$ as in Definition~\ref{symmetric}, for $\alpha\in\R\setminus\Lambda$. Computing $C^\star_{\alpha,p}$ and identifying the equality cases can be reduced to the case of a Gagliardo-Nirenberg interpolation inequality on~$\R$: see Appendix~\ref{App00}. Up to translation and multiplication by a constant, any optimizer among symmetric spinors satisfies the form of Definition~\ref{symmetric} with
\be{eq:minsymm}
u(s)=\big(\cosh(s/s_0)\big)^{-2/(p-2)}\quad\forall\,s\in\R\,,
\ee
for some $s_0>0$, which depends only on $\alpha$ and $p$. Moreover, symmetric spinors are of type $\mathcal H_0$ if $\alpha>-1/2$ and of type $\mathcal H_{-2}$ if $\alpha<-1/2$. Note that this is consistent with the symmetry between $(\alpha,p)$ and $(-\alpha-1,p)$, since the map $\phi\mapsto(\sigma\cdot\omega)\,\phi$ maps $\mathcal H_0$ to $\mathcal H_{-2}$.

\medskip\noindent The question raised above becomes:\\
\centerline{\emph{For which $\alpha$ and $p$ are the global optimizers for Inequality~\eqref{logarithmicSCKN} symmetric ?}}

\noindent Let us define the symmetry and symmetry breaking regions as subsets of $\mathbb R \times (2,6) \ni (\alpha,p) $ where the following holds: symmetry breaking means $C_{\alpha,p}<C^\star_{\alpha,p}$ while symmetry means $C_{\alpha,p}=C^\star_{\alpha,p}$ and all optimizers are then symmetric. In strong contrast with the scalar case (see Fig.~\ref{Fig:SCKN}), our main result goes as follows.
\begin{theorem}\label{Thm:main}
There are optimizers for~\eqref{logarithmicSCKN} for all $\alpha\in\R\setminus\Lambda$ and $p\in(2,6)$. Moreover, there are two functions $\alpha_\star$ and $\alpha^\star$ on $(2,6)$, taking values respectively in $(-1/2,0]$ and $[1/2,1)$, such that the following holds for all $\alpha\in[-1/2,\infty)$ and $2<p<6$:
\begin{enumerate}
\item[\rm(i)] symmetry holds if either $\alpha_\star(p)<\alpha<1/2$ or $1/2<\alpha<\alpha^\star(p)$,
\item[\rm(ii)] symmetry breaking holds if either $-\,1/2\le\alpha<\alpha_\star(p)$, or $\alpha>\alpha^\star(p)$.
\end{enumerate}
Symmetry and symmetry breaking regions are invariant under the reflection with respect to $\alpha=-\,1/2$, i.e., the transformation $(\alpha,p)\mapsto\big(-(\alpha+1),p\big)$, so that the range $\alpha<-1/2$ is also covered. The functions $\alpha_\star$ and $\alpha^\star$ are such that
\begin{multline*}
-\,\frac12\le\max\left\{\alpha_1(p),\alpha_2(p)\right\}\le\alpha_\star(p)\le\min\left\{0,\alpha_0(p)\right\}\\\mbox{and}\quad\min\left\{\alpha_3(p),\alpha_4(p)\right\}\le\alpha^\star(p)\le\alpha_5(p)
\end{multline*}
where
\begin{align*}
&\alpha_0(p):=\tfrac{4-\sqrt{-3\,p^2+20\,p-12}}{2\,(p-2)}\,,\quad&
&\alpha_1(p):=\tfrac{2\,p^2-7\,p-6-\sqrt{3\,(19\,p^2-20\,p+12)}}{4\,p\,(p-2)}\,,\\
&\alpha_2(p):=\tfrac{p^2-2\,p+24-2\,\sqrt{3\,(11\,p^2-4\,p+12)}}{2\,(p-2)\,(p+6)}\,,\quad&
&\alpha_3(p):=\tfrac{6-p}{2\,(p-2)}\,,\\
&\alpha_4(p):=\tfrac{3\,p-10+\sqrt{-3\,p^2+20\,p-12}}{4\,(p-2)}\,,\quad&
&\alpha_5(p):=\tfrac{2\,p^2+5\,p-30+\sqrt{3\,(19\,p^2-20\,p+12)}}{4\,(p-2)\,(p+3)}\,.
\end{align*}
If $\alpha \in(\alpha_\star(p),1/2)\cup(1/2,\alpha^\star(p))$, then $C_{\alpha,p}=C_{\alpha,p}^\star$ and all optimizers are symmetric.
\end{theorem}
\noindent For $\alpha=\alpha_\star(p)$ or $\alpha=\alpha^\star(p)$, then we have $C_{\alpha,p}=C_{\alpha,p}^\star$ and there is a symmetric optimizer. Whether in this case \emph{all} optimizers are symmetric or not is an open question.
In Fig.~\ref{Fig:SCKNlog} below, we synthetize the results of Theorem~\ref{Thm:main}.
\begin{figure}[hb]
\includegraphics[width=14cm]{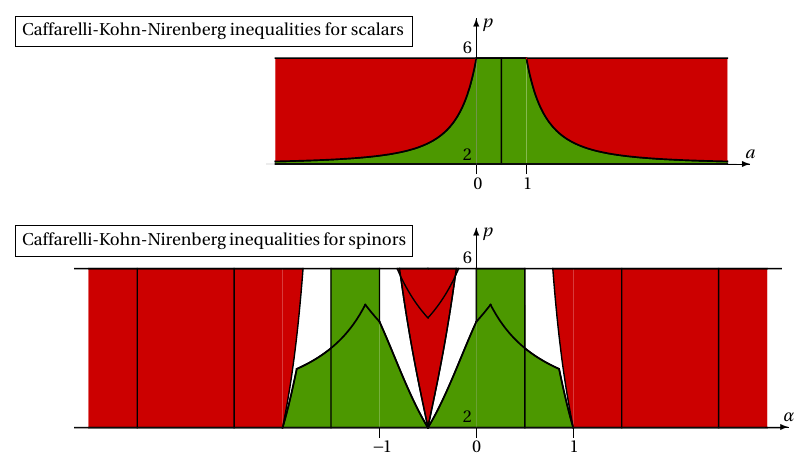}
\caption{\small\label{Fig:SCKNlog} Inequality~\eqref{logarithmicSCKN} for spinors and, for sake of comparison, its counterpart for scalars. The threshold between symmetry and symmetry breaking is determined by the functions $p\mapsto\alpha_\star(p)$ and $p\mapsto\alpha^\star(p)$ which take values in the white areas of the strip $2\le p\le6$. The curves corresponding to the \emph{ansatz} for symmetry breaking and those which determine the symmetry regions will be made clear in the proofs; also see Figs.~\ref{Fig:SCKNlog-Detail-1} and~\ref{Fig:SCKNlog-Detail-2}.}
\end{figure}

\begin{remark}\label{translation}
The equivalence between~\eqref{SCKN} and~\eqref{logarithmicSCKN} shows that Theorem~\ref{Cor:main} is equivalent to the first part of Theorem~\ref{Thm:main} via
$$
\overline\alpha_\star(\beta-\alpha) = \alpha_\star\(\frac{6}{1-2\,\alpha+2\,\beta}\) \,,
\quad
\overline\alpha^\star(\beta-\alpha) = \alpha^\star\(\frac{6}{1-2\,\alpha+2\,\beta}\) \,.
$$
This allows us to translate the bounds on $\alpha_\star$ and $\alpha^\star$ into bounds on $\overline\alpha_\star$ and $\overline\alpha^\star$ with the drawback that expressions are more complicated when expressed in terms of $\alpha$ and $\beta$.
\end{remark}

The reader familiar with the large literature on the sharp constant in~\eqref{CKN} will probably recognize the various strategies that we employ to prove symmetry or symmetry breaking. Looking closely at the details of how these strategies are implemented here, however, one finds substantial differences and, as a general rule, the present spinorial case is considerably more involved than the scalar case. A posteriori, this is not surprising, of course, given the more complicated shape of~\eqref{SCKN} compared to~\eqref{CKN}.

More specifically, in Section~\ref{Sec:firstsymmetry} we employ a monotonicity argument that is similar in spirit to an argument that appears in~\cite{MR2508844,MR2560127}. In contrast to the scalar case, when applying this argument for $\alpha\geq -1/2$ one has to distinguish the cases where $\alpha<1/2$ and $\alpha>1/2$. In Section~\ref{Sec:secondsymmetry} we employ an argument that is based on Gagliardo-Nirenberg interpolation inequalities on the sphere, inspired by~\cite{MR2905786}. The corresponding Gagliardo-Nirenberg inequality is new and might be interesting on its own. When adapting the strategy from~\cite{MR2905786} we overcome the difficulties from the spinorial nature through a judicious application of Minkowski's inequality. Finally, in Section~\ref{Sec:symmetrybreaking} we derive a symmetry breaking region by analyzing the stability of symmetric solutions in the spirit of~\cite{MR1973285}. Here again a subtlety appears, since the spinorial inequality~\eqref{SCKN} is intrinsically complex and the second variation operator is real-linear but not complex-linear. We manage to find a nontrivial symmetry breaking region, but in contrast to the scalar case, we are not able to determine the complete instability region\footnote{See the footnote in Remark~\ref{Rem:destabilization}.}.

One reason for these differences between the spinorial and the scalar case is the following. In logarithmic variables we find the angular operator $(\sigma\cdot L - \alpha +1/2)^2$; see Lemma~\ref{Lem:chvar}. Normalizing this operator so that it vanishes on constant spinors, we arrive at the effective angular operator $\mathcal K_\alpha =(\sigma\cdot L)^2 + (1-2\,\alpha)\,\sigma\cdot L$. Repeating the same in the scalar case we arrive first at the operator $L^2 +(\alpha -1/2)^2$ and then, after normalization, at the effective angular operator $\mathcal K^{\rm scalar} = L^2$, which \emph{is independent of $\alpha$}. Not only does $\mathcal K_\alpha$ depend on $\alpha$, but this dependence is also non-monotone, since $\sigma\cdot L$ is unbounded both from above and from below. Dealing with the $\alpha$-dependence of the effective angular operator is the cause of many of the difficulties that we face in this work and is one of the main new obstacles that we overcome in this paper.

Despite these differences, however, it is perhaps an interesting observation that the strategies that have previously been employed for the scalar~\eqref{CKN} problem are not entirely linked to the specifics of the problem. They are more robust than one might have imagined and may have a wider applicability.

Needless to say that the results obtained in this paper are not sharp in the sense that they do not cover the full parameter regime. Giving a complete description of the symmetry and symmetry breaking regions is an \emph{open problem}.

\medskip The paper is organized as follows. In Section~\ref{Sec:Hardy} we discuss the reformulation in logarithmic coordinates and consider, in particular, the limit case $p=2$, where Inequality~\eqref{SCKN} becomes a Hardy-like inequality for which the best constant can be computed explicitly. Existence of optimizers is proved if $\alpha<\beta<\alpha+1$ in Section~\ref{Sec:existence}. Sections~\ref{Sec:firstsymmetry} and~\ref{Sec:secondsymmetry} provide us with symmetry results for~\eqref{logarithmicSCKN}, while in Section~\ref{Sec:symmetrybreaking} we find regions where the symmetric solutions are linearly unstable. The proof of Theorem~\ref{Thm:main} is established in Section~\ref{Sec:proofmain}.

\medskip\noindent{\bf Notation.} We denote the canonical inner product of $\psi_1$, $\psi_2\in\C^2$ both by $\lrangle{\psi_1}{\psi_2}$ and by $\psi_1^*\,\psi_2$, according to typographical convenience. The corresponding norm is denoted by $|\psi| = \sqrt{\langle\psi,\psi\rangle}= \sqrt{\psi^*\psi}$.

\section{Changes of coordinates and preliminary estimates}\label{Sec:Hardy}

We begin with a computation that is crucial for much that follows.
\begin{proposition}\label{basiccomp}
For $\psi\in\mathcal D_{\alpha,\beta}$, setting $\psi(x)=|x|^{\alpha-1/2}\,\Phi(x)$, we have
\be{Id:basiccomp}
\ird{\frac{|\sigma\cdot\nabla\psi|^2}{|x|^{2\,\alpha}}}
= \ird{\frac{|\partial_r\Phi|^2}{|x|}}+\ird{\frac{\left|\(\sigma\cdot L-\alpha+\tfrac12\)\Phi\,\right|^2}{|x|^2}}\,.
\ee
\end{proposition}
\bigskip\noindent Here we use the notation $r=|x|$ and $\partial_r\psi=\frac xr\cdot\nabla\psi$. On $\R^3$, the angular momentum is defined by
\[
L:=x\wedge(-i\,\nabla)\,.
\]
This definition extends to $\R^3$ the operator $L$ acting on spinors on~$\Sp$, as in Section~\ref{Sec:Introduction}.

\begin{proof}
For any $x\in\R^3\setminus\{0\}$, we compute
\begin{multline*}
\frac1{|x|^{2\,\alpha}}\,(\sigma\cdot\nabla\psi)^*\,(\sigma\cdot\nabla\psi)=\frac1{|x|}\,(\sigma\cdot\nabla\Phi)^*\,(\sigma\cdot\nabla\Phi)+\frac1{|x|^3}\(\alpha-\tfrac12\)^2\,|\Phi|^2\\
+\frac2{|x|^3}\(\alpha-\tfrac12\)\,\mathrm{Re}\((\sigma\cdot\nabla\Phi)^*\,(\sigma\cdot x)\,\Phi\)\,.
\end{multline*}
Since $\nabla\cdot(x\,|x|^{-3})=0$ and $(\sigma\cdot\mathsf a)\,(\sigma\cdot\mathsf b)=\mathsf a\cdot\mathsf b+i\,\sigma\cdot(\mathsf a\wedge\mathsf b)$, we learn that
\[
\int_{\R^3}(\sigma\cdot \nabla\Phi)^*\,(\sigma\cdot x)\,\Phi\,\frac{dx}{|x|^3}=-\ird{\Phi^*\,(\sigma\cdot\nabla)\(\sigma\cdot\frac x{|x|^3}\)\,\Phi}
=-\int_{\R^3}\Phi^*\,(\sigma\cdot L)\,\Phi\,\frac{dx}{|x|^3}\,.
\]
Hence
\begin{multline*}
\ird{\frac{|\sigma\cdot\nabla\psi|^2}{|x|^{2\,\alpha}}}=\ird{\frac{|\sigma\cdot\nabla\Phi|^2}{|x|}}+\(\alpha-\tfrac12\)^2\ird{\frac{|\Phi|^2}{|x|^3}}\\
-2\(\alpha-\tfrac12\)\mathrm{Re}\(\ird{\frac{\Phi^*\,(\sigma\cdot L)\,\Phi}{|x|^3}}\)\,.
\end{multline*}
We learn from~\cite[p.~125]{MR1219537} that
\[
\sigma\cdot\nabla=\(\sigma\cdot\tfrac xr\)\partial _r-\tfrac{1}r\(\sigma\cdot\tfrac xr\)\,(\sigma\cdot L)
\]
and use it to compute
\[
\ird{\frac{|\sigma\cdot\nabla\Phi|^2}{|x|}}=\ird{\frac{|\partial_r\Phi|^2}{|x|}}+\ird{\frac{\left|\sigma\cdot L\,\Phi\right|^2}{|x|^3}}
\]
using the fact that
\[
\int_{\R^3}(\partial_r\Phi)^*\,(\sigma\cdot L)\,\Phi\,\frac{dx}{|x|^2}=0\,.
\]
Collecting estimates, we obtain
\begin{multline*}
\ird{\frac{|\sigma\cdot\nabla\psi|^2}{|x|^{2\,\alpha}}}=\ird{\frac{|\partial_r\Phi|^2}{|x|}}+\ird{\frac{\left|\sigma\cdot L \,\Phi\right|^2}{|x|^3}}\\
+\(\alpha-\tfrac12\)^2\ird{\frac{|\Phi|^2}{|x|^3}}-2\(\alpha-\tfrac12\)\ird{\frac{\Phi^*\,(\sigma\cdot L)\,\Phi}{|x|^3}}\,.
\end{multline*}
This proves the claimed equality.
\end{proof}

A first consequence of Proposition~\ref{basiccomp} is the computation of the sharp constant $\mathcal C_{\alpha,\beta}$ in the limiting case $\beta=\alpha+1$. This corresponds to $p=2$ according to~\eqref{param}, so Inequality~\eqref{SCKN} becomes a weighted Hardy-type inequality for spinors
\be{Hardy}
\int_{\R^3}\frac{|\sigma\cdot\nabla\psi(x)|^2}{|x|^{2\,\alpha}}\,dx\ge\mathcal C_{\alpha,\alpha+1}\int_{\R^3}\frac{|\psi(x)|^2}{|x|^{2\,\alpha+2}}\,dx \,,
\ee
which generalizes~\eqref{Hardy-Dirac}.
\begin{proposition}\label{Prop:Hardy}
The optimal constant in Inequality~\eqref{Hardy} is given by
\be{minmin}
\mathcal C_{\alpha,\alpha+1}=\min_{k\in\Z\setminus\{-1\}}\big(k-\alpha+\tfrac12\big)^2\,.
\ee
The constant $\mathcal C_{\alpha,\alpha+1}$ is positive if $\alpha\not\in\Lambda$, while $\mathcal C_{\alpha,\alpha+1}=0$ if $\alpha\in\Lambda$, with $\Lambda$ given by~\eqref{Lambda}. Moreover, we have an improved inequality given by
\be{Hardy:improved}
\ird{\frac{|\sigma\cdot\nabla\psi|^2}{|x|^{2\,\alpha}}}\ge\ird{\frac1{|x|^{2\,\alpha}}\left|\partial_r\psi-\(\alpha-\tfrac12\)\frac\psi{|x|}\right|^2}+ \mathcal C_{\alpha,\alpha+1}\ird{\frac{|\psi|^2}{|x|^{2\,\alpha+2}}}\,.
\ee
\end{proposition}
\noindent The scalar analogue of~\eqref{Hardy} is well known; see e.g.~Opic and Kufner's book on Hardy inequalities~\cite{MR1069756}.

\begin{proof}
Inequality~\eqref{Hardy:improved} with the constant $\mathcal C_{\alpha,\alpha+1}$ replaced by $\min_{k\in\Z\setminus\{-1\}}(k-\alpha+1/2)^2$ follows by rewriting the right-hand side of~\eqref{Id:basiccomp} in terms of $\psi$ and observing that the spectrum of $\,\(\sigma\cdot L-\alpha+1/2\)^2\,$ is the set $\big\{(k-\alpha+1/2)^2\,:\,k\in\Z\setminus\{-1\}\big\}$. Spectral properties of $\sigma\cdot L$ and related operators are further discussed in Appendix~\ref{App5}. This proves $\mathcal C_{\alpha,\alpha+1}\geq \min_{k\in\Z\setminus\{-1\}}(k-\alpha+1/2)^2$.

Let us call $k_0$ the integer $k$ that realizes the minimum in~\eqref{minmin}. Still with the notations of Proposition~\ref{basiccomp}, optimality in~\eqref{Hardy} is obtained by taking the limit as $n\to+\infty$ with the test function $\Phi_n(x)=g_n(|x|)\,\chi_{k_0}(x/|x|)$, where the function~$g_n$ is appropriately chosen, with compact support in $(0,+\infty)$ and the property that $g_n(r)=1$ for any $r\in(1/n,n)$.
\end{proof}

The second consequence of Proposition~\ref{basiccomp} is the rewriting of Inequality~\eqref{SCKN} in logarithmic coordinates, as was done in~\cite{MR1794994} for~\eqref{CKN}.
\begin{proof}[Proof of Lemma~\ref{Lem:chvar}] Let $\phi(s, \omega)=\Phi(r,\omega)$ with $s=\log r$, so that, in particular, $\partial_r\Phi = r^{-1}\,\partial_s\phi$. We recall that $d\omega$ is defined to be a probability measure, which gives rise to a factor of~$4\,\pi$. With this change of variables, we obtain
\begin{align*}
&\ird{\frac{|\partial_r\Phi|^2}{|x|}} = \ird{\frac1{|x|^3}\,|\partial_s\phi(s,\omega)|^2} = 4\,\pi \ic{|\partial_s\phi|^2}\,,\\
&\ird{\frac{\left|\(\sigma\cdot L-\alpha+\tfrac12\)\Phi\,\right|^2}{|x|^2}}= 4\,\pi \ird{\left|\(\sigma\cdot L-\alpha+\tfrac12\)\phi\,\right|^2}\,,\\
&\int_{\R^3}\frac{|\Phi|^p}{|x|^3}\,dx = 4\,\pi \ic{|\phi|^p}\,.
\end{align*}
With $\psi(x)=|x|^{\alpha-1/2}\,\Phi(x)$ as in Proposition~\ref{basiccomp}, so that in particular
\[
\frac{|\psi|^p}{|x|^{\beta\,p}}=\frac{|\Phi|^p}{|x|^3}\,,
\]
this proves that~\eqref{SCKN} written for $\psi$ on $\R^3$ is equivalent to~\eqref{logarithmicSCKN} written for $\phi$ on $\R\times\Sp$, with same optimal constants.\end{proof}

In order to prove Proposition~\ref{validity}, we observe that the left-hand side of~\eqref{logarithmicSCKN} is coercive in $\mathrm H^1(\R\times\Sp,\C^2)$ for all $\alpha\not\in\Lambda$ because of the following estimate for spinors on $\Sp$.
\begin{lemma}\label{Lem:H1}
For every $\alpha\not\in\Lambda$, there exists $\varepsilon(\alpha)\in(0,1]$ such that
\[
\is{\big|\(\sigma\cdot L-\alpha+\tfrac12\)\chi\big|^2}\ge\varepsilon(\alpha)\is{\(|L\,\chi|^2+|\chi|^2\)}\quad\forall\,\chi\in\mathrm H^1(\Sp)\,.
\]
\end{lemma}
\begin{proof}
Since $\mathrm{Sp}(\sigma\cdot L)=\big\{k\in\Z\,:\,k\neq -\,1\big\}$ and\,$\mathrm{Sp}(L^2)=\big\{\ell\,(\ell+1)\,:\,\ell\in\N\big\}$ (using the notation $\N=\{0,1,2,\ldots\}$) with the relations $k=\ell$ if $k\ge0$ and $k=-\,\ell-1$ if $k\le -2$, on $\mathcal H_k$ we have $(\sigma\cdot L)=k$ and $L^2=k\,(k+1)$. A spectral decomposition of $\chi=\sum_{k\in\Z\setminus\{-1\}}\chi_k$ with $\chi_k\in\mathcal H_k$ and $c_k(\chi)=\|\chi_k\|_{\mathrm L^2(\Sp,d\omega)}^2\ge0$ shows that 
\[
\varepsilon(\alpha)=\inf_{\chi\in\mathrm H^1(\Sp)\setminus\{0\}}\frac{\sum_{k\in\Z\setminus\{-1\}}\(k-\alpha+\tfrac12\)^2\,c_k(\chi)}{\sum_{k\in\Z\setminus\{-1\}}\big(k\,(k+1)+1\big)\,c_k(\chi)}=\min_{k\in\Z\setminus\{-1\}}\frac{\(k-\alpha+\tfrac12\)^2}{k\,(k+1)+1}
\]
is positive for all $\alpha\not\in\Lambda$.
\end{proof}

\begin{proof}[Proof of Proposition~\ref{validity}] If $\alpha\not\in\Lambda$, then we have
\[
\ic{\(|\partial_s\phi|^2+\big|\(\sigma\cdot L-\alpha+\tfrac12\)\phi\big|^2\)}\ge\varepsilon(\alpha)\ic{\(|\partial_s\phi|^2+|L\,\phi|^2+|\phi|^2\)}
\]
by Lemma~\ref{Lem:H1}. The above right-hand side is bounded from below by $ \varepsilon(\alpha)\,C_p\,\nrmc\phi p^2$, where $C_p$ is the optimal constant in the embedding $\mathrm H^1(\R\times\Sp)\hookrightarrow\mathrm L^p(\R\times\Sp)\,$. This proves $\mathcal C_{\alpha,\beta}\ge \varepsilon(\alpha)\,C_p>0$.

If $\alpha\in\Lambda$, there is a spinor $\chi\not\equiv 0$ on $\Sp$ with $\(\sigma\cdot L-\alpha+1/2\)\chi=0$. For any $\beta\in[\alpha,\alpha+1]$ consider $p$ determined by~\eqref{param} and a family of test functions $\phi_\varepsilon(s,\omega)=\varepsilon^{1/p}\,u(\varepsilon\,s)\,\chi(\omega)$ such that $\nrmc{\phi_1}p=1$. Then $\nrmc{\phi_\varepsilon}p=1$ is independent of $\varepsilon>0$, while
\[
\lim_{\varepsilon\to0_+}\ic{\(|\partial_s\phi_\varepsilon|^2+\big|\(\sigma\cdot L-\alpha+\tfrac12\)\phi_\varepsilon\big|^2\)}=\lim_{\varepsilon\to0_+}\nrmc{\partial_s\phi_\varepsilon}2^2=0 \,.
\]
This shows that $\mathcal C_{\alpha,\beta}=0$ for $\alpha\in\Lambda$ and any $\beta\in[\alpha,\alpha+1]$.
\end{proof}

As a consequence of Lemma~\ref{Lem:H1}, the space $\mathcal D_{\alpha,\beta}$ in~\eqref{SCKN} is transformed into $\mathrm H^1(\R\times\Sp,\C^2)$ if $\alpha\not\in\Lambda$ and the left-hand side in~\eqref{logarithmicSCKN} is equivalent to the standard norm of $\mathrm H^1(\R\times\Sp,\C^2)$.

\section{Existence of optimizers}\label{Sec:existence}

Here we prove the following existence result.
\begin{theorem}\label{Thm:Existence}
For all $p\in (2, 6)$ and all $\alpha\not\in\Lambda$, the best constant $C_{\alpha,p}$ in Inequality~\eqref{logarithmicSCKN} is attained in $\mathrm H^1(\R\times\Sp,\C^2)\setminus\{0\}$.
\end{theorem}
\begin{proof} Here we follow very closely the arguments in the proof of~\cite[Theorem 1.2~(i)]{MR1794994} and~\cite{MR2966111}. Let us consider a sequence $(\phi_n)_{n\in\N}$ in $\mathrm H^1(\R\times\Sp,\C^2)$ such that \hbox{$\nrmc{\phi_n}p=1$} for any $n\in\N$ and
\be{Opt:Seq}
\lim_{n\to\infty}\(\ic{|\partial_s\phi_n|^2}+\ic{\big|\(\sigma\cdot L-\alpha+\tfrac12\)^2\phi_n\big|^2}\)=C_{\alpha,p}\,.
\ee
Because of Lemma~\ref{Lem:H1}, we know that $(\phi_n)_{n\in\N}$ is bounded in $\mathrm H^1(\R\times\Sp,\C^2)$.

Take any fixed $r>0$ and consider $A_r(\sigma):=(\sigma-r,\sigma+r)\times\Sp$ for $\sigma\in\R$. We learn from~\cite[Lemma~4.1]{MR1794994}, which is adapted from~\cite[Lemma~I.1,~p.~231]{MR0778974}, that any bounded sequence $(f_n)_{n\in\N}$ in $\mathrm H^1(\R\times\Sp,\R)$ such that
\[
\lim_{n\to\infty}\,\sup_{\sigma\in\R}\,\iint_{A_r(\sigma)}|f_n|^p\,ds\,d\omega=0
\]
satisfies $\lim_{n\to\infty}\nrmc{f_n}p=0$. The same result holds for spinors. Since
$$
\lim_{n\to\infty}\nrmc{\phi_n}p=1\neq0\,,
$$
we deduce that there exists a constant $C\in(0,1]$ and $\sigma_n\in\R$ such that
\be{supintegral}
\int_{A_r(\sigma_n)}|\phi_n|^p\,dy\ge C\quad\mbox{for $n$ large enough}\,.
\ee
Define $\widetilde\phi_n(s,\omega):= \phi_n(s-\sigma_n,\omega)$. Up to the extraction of a subsequence, $(\widetilde\phi_n)_{n\in\N}$ converges weakly in $\mathrm H^1(\R\times\Sp,\C^2)$, strongly in $\mathrm L^p_{\rm loc}(\R\times\Sp,\C^2)$ and almost everywhere to some spinor~$\phi$.
Therefore, by~\eqref{supintegral}, $\phi\equiv0$ is forbidden and $\mathsf x:=\nrmc\phi p^p\in(0,1]$. By the almost everywhere convergence and the Brezis-Lieb lemma (see~\cite[Theorem 1]{MR0699419}) we find
\[
\lim_{n\to\infty}\nrmc{\widetilde\phi_n}p^p=\nrmc\phi p^p+\lim_{n\to\infty}\nrmc{\widetilde\phi_n-\phi}p^p
\]
which, in view of the normalization of $\widetilde\phi_n$, implies that
\[
\lim_{n\to\infty}\nrmc{\widetilde\phi_n-\phi}p^2 = (1-\mathsf x)^{2/p} \,.
\]
Meanwhile, by the weak convergence in $\mathrm H^1(\R\times\Sp,\C^2)$ we find
\[\begin{array}{l}
\lim_{n\to\infty}\(\nrmc{\partial_s\widetilde\phi_n}2^2+\nrmc{\(\sigma\cdot L-\alpha+\tfrac12\)\widetilde\phi_n}2^2\)\\[6pt]
\hspace*{2cm}=\nrmc{\partial_s\phi}2^2+\nrmc{\(\sigma\cdot L-\alpha+\tfrac12\)\phi}2^2\\[6pt]
\hspace*{3cm}+\lim_{n\to\infty}\(\nrmc{\partial_s\phi-\partial_s\widetilde\phi_n}2^2+\nrmc{\(\sigma\cdot L-\alpha+\tfrac12\)\,(\phi-\widetilde\phi_n)}2^2\)\,.
\end{array}\]
We bound the right side from below using Inequality~\eqref{logarithmicSCKN} for $\phi$ and for $\phi - \widetilde\phi_n$. Recalling~\eqref{Opt:Seq}, we obtain
\[
C_{\alpha, p}\ge C_{\alpha,p}\(\mathsf x^{2/p}+(1-\mathsf x)^{2/p}\)\,.
\]
From this we deduce that $\mathsf x=1$. Hence $(\widetilde\phi_n)_{n\in\N}$ converges to $\phi$ strongly in $\mathrm L^p(\R\times\Sp,\C^2)$. Moreover, in the application of Inequality~\eqref{logarithmicSCKN} for $\phi-\widetilde\phi_n$ we must have asymptotically equality, which implies that $(\widetilde\phi_n)_{n\in\N}$ converges to $\phi$ also strongly in $\mathrm H^1(\R\times\Sp,\C^2)$. This proves that $\phi\not\equiv 0$ is an optimal function for~\eqref{logarithmicSCKN}.
\end{proof}

\section{A first symmetry result and a monotonicity property}\label{Sec:firstsymmetry}

We begin by proving symmetry of optimizers for $\alpha=0$.
\begin{lemma}\label{Lem:alpha=0alt}
If $\alpha=0$ and $\beta\in[0,1)$, any optimizer of~\eqref{SCKN} is symmetric of type $\mathcal H_0$.
\end{lemma}
\begin{proof}
We notice that, for any $\psi=\begin{psmallmatrix}\psi_1\\\psi_2\end{psmallmatrix}\in\mathrm H^1(\R^3,\C^2)$, we have
$$
\int_{\R^3} |\sigma\cdot\nabla\psi|^2\,dx = \int_{\R^3} |\nabla\psi|^2\,dx \,,
$$
where we use the notation $|\nabla\psi|^2 = \sum_{j=1}^3 \sum_{k=1}^2|\partial_{\!j}\psi_k|^2$. We now use the scalar inequality~\eqref{CKN} with $b=\beta$ and note that $C_{0,\beta,3}=\mathcal C^\star_{0,\beta}$ because optimizers of~\eqref{CKN} are radial according to~\cite{MR1223899, MR1731336, MR2560127} if $a\in[0,1/2)$ and $b\in[a,a+1)$, eventually up to translation if $(a,b)=(0,0)$. In this way we obtain for $k=1$, $2,$
$$
\int_{\R^3} |\nabla\psi_k|^2\,dx \geq \mathcal C^\star_{0,\beta} \left( \int_{\R^3} \frac{|\psi_k(x)|^p}{|x|^{\beta\,p}}\,dx \right)^{2/p}.
$$
Thus, we have shown that
$$
\int_{\R^3} |\sigma\cdot\nabla\psi|^2\,dx
\geq \mathcal C^\star_{0,\beta} \left( \big\||x|^{-\beta}\,\psi_1\big\|_{\mathrm L^p(\R^3)}^2 + \big\||x|^{-\beta}\,\psi_2\big\|_{\mathrm L^p(\R^3)}^2 \right)\,.
$$
Since, by the triangle inequality in $\mathrm L^{p/2}(\R^3, \C^2)$,
\begin{equation*}
\begin{aligned}
\big\||x|^{-\beta}\,\psi_1\big\|_{\mathrm L^p(\R^3)}^2 + \big\||x|^{-\beta}\,\psi_2\big\|_{\mathrm L^p(\R^3)}^2
& = \big\||x|^{-2\,\beta}\,|\psi_1|^2\big\|_{\mathrm L^{p/2}(\R^3)} + \big\||x|^{-2\,\beta}\,|\psi_2|^2 \big\|_{\mathrm L^{p/2}(\R^3)} \\
& \geq \big\||x|^{-2\,\beta}\,\big(|\psi_1|^2 + |\psi_2|^2\big)\big\|_{\mathrm L^{p/2}(\R^3)} = \big\||x|^{-\beta}\,\psi \big\|_{\mathrm L^p(\R^3)}^2 \,,
\end{aligned}
\end{equation*}
we have shown Inequality~\eqref{SCKN} with constant $\mathcal C_{0,\beta}=\mathcal C^\star_{0,\beta}$.

Now let us show the stronger statement that every optimizer $\psi$ for the inequality is symmetric. For clarity we first consider the case $\beta>0$. The scalar inequality~\eqref{CKN} needs to be saturated for both scalar functions $\psi_1$ and $\psi_2$, so there are $c_k\in\C$ and $a_k>0$ such that
$$
\psi_k(x) = c_k\,v_\star(x/a_k)\quad\mbox{with}\quad\,k=1\,,\;2\,.
$$
Here $v_\star$ is an explicit, radial, positive function, and we can assume that $a_1=1$ without loss of generality. We used here the characterization of cases of equality for~\eqref{CKN} for $a=0$. Equality in the triangle inequality in $\mathrm L^{p/2}(\R^3, \C^2)$ implies that either $|x|^{-2\,\beta}\,|\psi_2|^2=0$ or that there is a $c\geq0$ such that $|x|^{-2\,\beta}\,|\psi_1|^2=c\,|x|^{-2\,\beta}\,|\psi_2|^2$. Clearly, this is the same as either having $c_2=0$ or $|c_1|\,v_\star(x/a_1) = \sqrt c\,|c_2|\,v_\star(x/a_2)$ for all $x$. By the properties of $v_\star$ we see that the second option is equivalent to $a_1=a_2$ and $|c_1|=\sqrt c\,|c_2|$. Thus, in either case we have
$$
\psi(x) = v_\star(x) \, \chi\quad\forall\,x\in\R^3\,,
$$
where $\chi := \begin{psmallmatrix} c_1 \\ c_2 \end{psmallmatrix}\in\C^2$ is a constant spinor. This proves the symmetry of any optimizer when $\beta>0$. The argument in the case $\beta=0$ is similar, but one needs to take translations into account as well. We omit the details.
\end{proof}
\begin{remark}\label{Rem:ComplexCKN}
In the previous proof we assumed implicitly that the scalar inequality~\eqref{CKN} holds for complex functions with the same constant as for real functions. This is well known and can be proved with essentially the same argument as in the above proof, where the real inequality is applied to $\Re u$ and $\Im u$. Alternatively, one can repeat the proof with four real functions $\Re\,\psi_1$, $\Im\,\psi_1$, $\Re\,\psi_2$, $\Im\,\psi_2$, based only on the real version of Inequality~\eqref{CKN}.
\end{remark}

The result of Lemma~\ref{Lem:alpha=0alt} can be rephrased into an equivalent result in logarithmic variables: for any $p\in(2,6]$, we have $C_{0,p}=C^\star_{0,p}$ and if $p\in(2,6)$ optimizers of~\eqref{logarithmicSCKN} are given by~\eqref{eq:minsymm} up to translation in $s$ and multiplication by a constant spinor in $\C^2$. Among symmetric functions, when $\alpha\geq-1/2$ Inequality~\eqref{logarithmicSCKN} is reduced to the one-dimensional Gagliardo-Nirenberg inequality
\be{OneD}
\nrmr{u'}2^2+\lambda\,\nrmr u2^2\ge K_{\lambda,p}\,\nrmr up^2
\ee
with $\lambda=\lambda(\alpha)=(\alpha-1/2)^2$, $K_{\lambda,p}=C^\star_{\alpha,p}$ and $p$ given by~\eqref{param}. See Appendix~\ref{App00} for details. The rescaling $s\mapsto\sqrt\lambda\,s$ shows that $K_{\lambda,p}=\lambda^{(p+2)/(2\,p)}\,K_{1,p}$. For any $t\ge1$, let
\[
\alpha_t:=\tfrac12-\big(\tfrac12-\alpha\big)\,t
\]
so that $\lambda(\alpha_t)=\lambda(\alpha)\,t^2$ and, as a consequence,
\be{star:t}
C^\star_{\alpha_t,p}=C^\star_{\alpha,p}\,t^{1+\frac2p}\,.
\ee
We refer to Proposition~\ref{symmetricproblem} in Appendix~\ref{App00} for more details. It turns out that we can also use a scaling in the $s=\log|x|$ direction in~\eqref{logarithmicSCKN} even for non-symmetric functions. Here is how one can argue, as in~\cite{MR2508844} or in~\cite[Lemma~4.1]{MR2560127} in the scalar case. Let us define the functional
\be{global-functional-cylinder}
\mathcal G_\alpha[\phi]:=\frac{\nrmc{\partial_s\phi}2^2+\big\|\big(\sigma\cdot L+\tfrac12-\alpha\big)\,\phi\|_{\mathrm L^2(\R\times\Sp)}^2}{\nrmc{\phi}p^2}\,.
\ee
\begin{proposition}\label{Prop:alonglines} Assume that either $(\alpha,p)\in(-1/2,1/2)\times(2,6)$ or $(\alpha,p)\in(1/2,3/2)\times(2,6)$. If one of the two conditions is satisfied:
\begin{enumerate}
\item[\rm(i)] either $C_{\alpha,p}<C^\star_{\alpha,p}$
\item[\rm(ii)] or there is a non-symmetric function $\phi\in\mathrm H^1(\R^3,\C^2)$ such that $\mathcal G_{\alpha}[\phi]=C_{\alpha,p}=C^\star_{\alpha,p}$
\end{enumerate}
then we have
\[
C_{\alpha',p}<C^\star_{\alpha',p}\quad\forall\,\alpha'\in[-1/2,\alpha)\quad\mbox{if}\quad\alpha<1/2\quad\mbox{and}\quad
\quad\forall\alpha'\in(\alpha,3/2)\quad\mbox{if}\quad\alpha>1/2\,.
\]
\end{proposition}
\noindent Here a non-symmetric $\phi$ means that $\phi$ is not symmetric of type $\mathcal H_0$ and consequently $\nrmc{\sigma\cdot L \,\phi}2\neq0$. Proposition~\ref{Prop:alonglines} combined with Lemma~\ref{Lem:alpha=0alt} extends for instance the range of symmetry to any $\alpha\in[0,1/2)$. However it has further consequences as we shall see in Corollary~\ref{Cor:green}.
\begin{proof} Assume that $\alpha\in(-1/2,1/2)$. For any $\phi\in\mathrm H^1(\R^3,\C^2)$, let $\phi_t(s,\omega):=\phi(t\,s,\omega)$ for any real constant $t>1$ and $\alpha_t:=\tfrac12-\big(\tfrac12-\alpha\big)\,t$, such that $\tfrac12-\alpha_t=\big(\tfrac12-\alpha\big)\,t$. Changing variables in~\eqref{global-functional-cylinder}, we obtain
\[
\mathcal G_{\alpha_t}[\phi_t]=t^{\frac2p-1}\,\frac{t^2\,\nrmc{\partial_s\phi}2^2+\big\|\big(\sigma\cdot L+\big(\tfrac12-\alpha\big)\,t\big)\,\phi\|_{\mathrm L^2(\R\times\Sp)}^2}{\nrmc{\phi}p^2}
\]
and, as a consequence,
\begin{multline*}
\nrmc{\phi}p^2\(\mathcal G_{\alpha_t}[\phi_t]-t^{1+\frac2p}\,\mathcal G_\alpha[\phi]\)\\
=t^{\frac2p-1}\,(1-t)\((1+t)\,\nrmc{ \sigma\cdot L \,\phi}2^2+t\,(1-2\,\alpha)\is{\lrangle\phi{\sigma\cdot L \,\phi}}\)\,.
\end{multline*}
Let $\phi=\sum_{k\in\Z\setminus\{-1\}}\phi_k$ be the spectral decomposition (see Appendix~\ref{App5}) of $\phi$ on the eigen\-spaces of $\sigma\cdot L$ such that $\sigma\cdot L\,\phi_k=k\,\phi_k$. We have
\[
\mathcal R:=(1+t)\,\nrmc{\sigma\cdot L \,\phi}2^2+t\,(1-2\,\alpha)\is{\lrangle\phi{ \sigma\cdot L \,\phi}}=\sum_{k\in\Z\setminus\{-1\}}\nrmc{\phi_k}2^2\,h(k,t)
\]
with $h(k,t):=(1+t)\,k^2+(1-2\,\alpha)\,k\,t$. The assumption $\alpha\in(-1/2,1/2)$ implies that $h(k,t)$ is non-negative if $(k,t)\in (\Z\setminus\{-1\})\times(1,+\infty)$ and vanishes if and only if $k=0$. We deduce that $\mathcal R\ge0$ and therefore
\[
\mathcal G_{\alpha_t}[\phi_t]-t^{1+\frac2p}\,\mathcal G_\alpha[\phi]=t^{\frac2p-1}\,(1-t)\,\mathcal R\le0\quad\forall\,t>1\,.
\]
In case (i) let $\phi\in\mathrm H^1(\R^3,\C^2)$ be such that $C_{\alpha,p}\le\mathcal G_{\alpha}[\phi]<C^\star_{\alpha,p}$. It is then clear that $\mathcal R$ is positive because otherwise we have $\phi(s,\cdot)\in\mathcal H_0$ for a.e.~$s\in\R$, which would imply that~$\phi$ is symmetric and consequently $\mathcal G_{\alpha}[\phi]\ge C^\star_{\alpha,p}$, contradicting the choice of $\phi$. In case (ii), we apply the argument with the given non-symmetric function $\phi$ and not that this non-symmetry implies $\mathcal R>0$. Thus, in either case we find a function $\phi$ such that
\[
\mathcal G_{\alpha_t}[\phi_t]-t^{1+\frac2p}\,\mathcal G_\alpha[\phi]<0\quad\forall\,t>1\,.
\]
Now for any $\alpha'$ such that $-\,1/2\le\alpha'<\alpha<1/2$, let $t=\frac{1-2\,\alpha'}{1-2\,\alpha}>1$, i.e., $\alpha'=\alpha_t$. Altogether we obtain
\[
C_{\alpha',p}\le\mathcal G_{\alpha_t}[\phi_t]<t^{1+\frac2p}\,\mathcal G_{\alpha}[\phi]\le t^{1+\frac2p}\,C^\star_{\alpha,p}=C^\star_{\alpha',p} \,,
\]
where the equality follows from~\eqref{star:t}. This completes the proof if $\alpha\in(-1/2,1/2)$.

The case $\alpha\in(1/2,3/2)$ is similar, except that $\alpha_t:=\tfrac12-\big(\tfrac12-\alpha\big)\,t>\alpha$ for any $t>1$.
\end{proof}

As a consequence of Lemma~\ref{Lem:alpha=0alt} and Proposition~\ref{Prop:alonglines}, symmetry in~\eqref{SCKN} (and therefore also in~\eqref{logarithmicSCKN}) holds for any $\alpha\in[0,1/2)$. We can state a slightly better result as follows, in the spirit of~\cite[Theorem~1.2]{MR2560127}.
\begin{corollary}\label{Cor:green} There is a function $\alpha_\star:(2,6)\to[-1/2,0]$ such that, for any $p\in(2,6)$,
\begin{enumerate}
\item[\rm(i)] $C_{\alpha,p}<C^\star_{\alpha,p}$ if $-\,1/2\le\alpha<\alpha_\star(p)$ and $C_{\alpha,p}=C^\star_{\alpha,p}$ if $\alpha_\star(p)\le\alpha<1/2$,
\item[\rm(ii)] $C_{\alpha,p}<C^\star_{\alpha,p}$ if $-\,\big(1+\alpha_\star(p)\big)<\alpha\le-\,1/2$ and $C_{\alpha,p}=C^\star_{\alpha,p}$ if $\,-\,3/2<\alpha\le-\,\big(1+\alpha_\star(p)\big)$.
\end{enumerate}
Moreover, any optimizer is symmetric if $\alpha\in\big(-3/2,-\,\big(1+\alpha_\star(p)\big)\big)\cup\big(\alpha_\star(p),1/2\big)$.
\\[4pt]
Similarly, there is a function $\alpha^\star:(2,6)\to[1/2,3/2]$ such that, for any $p\in(2,6)$,
\begin{enumerate}
\item[\rm(i)] $C_{\alpha,p}<C^\star_{\alpha,p}$ if $\alpha^\star(p)<\alpha<3/2$ and $C_{\alpha,p}=C^\star_{\alpha,p}$ if $\,1/2<\alpha\le\alpha^\star(p)$,
\item[\rm(ii)] $C_{\alpha,p}<C^\star_{\alpha,p}$ if $\,-\,5/2<\alpha<-\,\big(1+\alpha^\star(p)\big)$ and $C_{\alpha,p}=C^\star_{\alpha,p}$ if $\,-\,\big(1+\alpha^\star(p)\big)\le\alpha<-\,3/2$.
\end{enumerate}
Moreover, any optimizer is symmetric if $\alpha\in\big(-\,\big(1+\alpha^\star(p)\big),-3/2\big)\cup\big(1/2,\alpha^\star(p)\big)$.
\end{corollary}
\begin{proof} It is easy to see from~Proposition~\ref{Prop:alonglines} that for a given $p\in(2,6)$, $\alpha_\star(p)$ can be defined as the infimum of the set $\big\{\alpha\in(-1/2,1/2)\,:\,C_{\alpha,p}=C^\star_{\alpha,p}\big\}\supset[0,1/2)$ where the inclusion follows from Lemma~\ref{Lem:alpha=0alt}. The existence of a non-symmetric optimizer for some $\alpha\in\big(\alpha_\star(p),1/2\big)$ would contradict Proposition~\ref{Prop:alonglines}, (ii). The results of (ii) for $\alpha\in(-1,-1/2)$ follow from (i) using~\eqref{logarithmicalt}. A similar discussion holds if $\alpha\in(-5/2,-3/2)\cup(1/2,3/2)$.\end{proof}

\section{A second symmetry result}\label{Sec:secondsymmetry}

The main result of this section in Theorem~\ref{Thm:symmetry1}, which gives a region where optimizers of~\eqref{logarithmicSCKN} are symmetric. An important ingredient in our proof is a Gagliardo-Nirenberg interpolation inequality for spinors on the sphere, which is interesting by itself and which is proved in the first subsection.

\subsection{A Gagliardo-Nirenberg interpolation inequality for spinors on the sphere}

Let us consider the operator
$$
\mathcal K_\alpha:=(\sigma\cdot L)^2+(1-2\,\alpha)\,\sigma\cdot L
$$
acting on angular spinors on $\Sp$ and define
\[
\mathsf m(\alpha):=\min\big\{1-\alpha,1+2\,\alpha\big\}
\]
such that $\mathsf m(\alpha)=1+2\,\alpha$ if $\alpha\in(-1/2,0]$ and $\mathsf m(\alpha)=1-\alpha$ if $\alpha\in[0,1)$, and
\be{q}
\mathsf q(\alpha):=4-\frac6\alpha\quad\mbox{if}\quad\alpha\in(-1/2,0)\,,\quad\mathsf q(0):=+\infty\quad\mbox{and}\quad\mathsf q(\alpha):=2+\frac2\alpha\quad\mbox{if}\quad\alpha\in(0,1)\,.
\ee
\begin{theorem}\label{thm:BecknerGenSpinor} Let $\alpha\in(-1/2,1)$ and $q\in(2,\infty)$. Then there is a constant $\mathcal B_{\alpha,q}>0$ such that for all $\Phi\in\mathrm H^1(\Sp,\C^2)$, we have
\be{Ineq:GenPoincare}
\is{\lrangle\Phi{\mathcal K_\alpha\,\Phi}}\ge \frac{\mathcal B_{\alpha,q}}{q-2}\(\nrms\Phi q^2-\nrms\Phi 2^2\)\,.
\ee
Moreover, if $q\leq \mathsf q(\alpha)$, then the inequality holds with $\mathcal B_{\alpha,q} = 2 \, \mathsf m(\alpha)$ and equality holds if and only if $\,\Phi\in\mathcal H_0$, that is, if and only if $\,\Phi$ is a constant spinor.
\end{theorem}
We have no reason to think that $\mathcal B_{\alpha,q} = 2 \, \mathsf m(\alpha)$ is optimal, even if $q\leq \mathsf q(\alpha)$. We recall that $d\omega$ is the uniform probability measure on $\Sp$. In particular, if $q>2$, by H\"older's inequality we have $\nrms\Phi q^2-\nrms\Phi 2^2\geq 0$ and then equality holds if and only if~$|\Phi|$ is constant.

Also notice that, by taking the limit as $q\to2_+$ in the right-hand side of~\eqref{Ineq:GenPoincare} and using $\mathcal B_{\alpha,q}=2 \, \mathsf m(\alpha)$ for $q$ sufficiently close to $2$, for any $\alpha\in(-1/2,1)$, we obtain the logarithmic Sobolev inequality
\[
\is{\lrangle\Phi{\mathcal K_\alpha\,\Phi}}\ge\frac12\,\mathsf m(\alpha)\is{|\Phi|^2\,\log\(\frac{|\Phi|^2}{\nrms\Phi 2^2}\)}\quad\forall\,\Phi\in\mathrm H^1(\Sp,\C^2)\,.
\]

We begin the proof of Theorem~\ref{thm:BecknerGenSpinor} by extending a well-known interpolation inequality obtained by Bidaut-V\'eron and V\'eron in~\cite{MR1134481} and Beckner in~\cite{MR1230930} for scalar functions on $\Sp$, to spinors, before adapting it to the operator $\mathcal K_\alpha$.
\subsubsection*{$\bullet$ An interpolation inequality for spinors.} Let us review Beckner's rewriting~\cite{MR1230930} of Lieb's sharp Hardy-Littlewood-Sobolev inequality~\cite{MR717827} for scalar functions. Decomposing $F\in\mathrm H^1(\Sp,\C)$ into spherical harmonics, that is, writing
\[
F = \sum_{\ell=0}^\infty F_\ell
\quad\text{with}\quad
L^2\,F_\ell=\ell\,(\ell+1)\,F_\ell \,,
\]
we can formulate the sharp Hardy-Littlewood-Sobolev inequality in the form
\be{Beckner}
\frac1{q-2}\(\nrms Fq^2-\nrms F2^2\)\le\sum_{\ell=1}^\infty\zeta_\ell(q)\is{|F_\ell|^2}
\ee
with
\[
\zeta_\ell(q):=\frac{\gamma_\ell\big(\tfrac2q\big)-1}{q-2}\quad\mbox{and}\quad\gamma_\ell(x):=\frac{\Gamma(x)\,\Gamma(\ell+2-x)}{\Gamma(2-x)\,\Gamma(x+\ell)}\,.
\]
We refer to~\cite[Ineq.~(19)]{MR1230930} and also~\cite[Ineq.~(1.6)]{MR3562947} for more details. Let us recall how~\eqref{Beckner} is proved in~\cite{MR1230930}. Consider the sharp Hardy-Littlewood-Sobolev inequality
written on the sphere
\be{nonspinor}
0\le \iint_{\Sp \times \Sp} \frac{\overline{F(\omega)}\,F(\omega')}{|\omega - \omega'|^{\lambda}}\,d\omega\,d \omega' \le C_p\,\Vert F\Vert_{\mathrm L^p(\Sp)}^2\,,
\ee
where $\lambda = 4/q$, $1\leq p <2$, $1/p+1/q=1$ and $C_p$ is the explicitly known sharp constant from~\cite{MR717827}. Computing the eigenvalues of the integral operator with kernel $|\omega-\omega'|^{-\lambda}$ by means of the Funk-Hecke formula, we see that
\begin{equation}
\label{eq:funkhecke}
\iint_{\Sp \times \Sp} \frac{\overline{F(\omega)}\,F(\omega')}{|\omega - \omega'|^{\lambda}}\,d\omega\,d \omega'
= C_p\,\sum_{\ell=0}^\infty \gamma_\ell\big(\tfrac2p\big)\,\Vert F_\ell\Vert_{\mathrm L^2(\Sp)}^2
\end{equation}
with the constants $\gamma_\ell(2/p)$ defined above. (The details of this computation can be found, for instance, in~\cite{MR2848628}.) Thus, Inequality~\eqref{nonspinor} is equivalent to
\[
\sum_{\ell=0}^\infty \gamma_\ell\big(\tfrac2p\big)\,\Vert F_\ell\Vert_{\mathrm L^2(\Sp)}^2 \le \Vert F\Vert_{\mathrm L^p(\Sp)}^2\,.
\]
By duality, noting that $\gamma_\ell(2/p)^{-1}=\gamma_\ell(2/q)$ if $p$ and $q$ are H\"older conjugates, we see that this is equivalent to the inequality
$$
\Vert F\Vert_{\mathrm L^q(\Sp)}^2 \le \sum_{\ell=0}^\infty \gamma_\ell\big(\tfrac2q\big)\,\Vert F_\ell\Vert_{\mathrm L^2(\Sp)}^2\,.
$$
The latter is a rewriting of~\eqref{Beckner}, thus concluding the proof as in~\cite{MR1230930}.

We can now generalize this method to the spinorial case. We decompose (angular) spinors $\Phi$ as
$$
\Phi = \sum_{\ell=0}^\infty \Phi_\ell
\quad\text{with}\quad
L^2\,\Phi_\ell = \ell\,(\ell+1)\,\Phi_\ell \,.
$$
\begin{lemma}\label{Lem:SpinorialS2interpolation}
Let $q\in(2,\infty)$. For any spinor $\Phi\in\mathrm H^1(\Sp,\C^2)$, we have
\be{spinorbeckner}
\frac1{q-2}\(\nrms\Phi q^2-\nrms\Phi 2^2\)\le\sum_{\ell=1}^\infty\zeta_\ell(q)\is{|\Phi_\ell|^2}
\ee
with equality if and only if
$$
\Phi(\omega) = (1- \zeta\cdot\omega)^{-2/q} \chi_0
$$
for some $\zeta\in\R^3$ with $|\zeta|<1$ and some $\chi_0\in\C^2$.
\end{lemma}
\begin{proof} For $\Phi \in \mathrm L^p(\Sp, \C^2)$ we have
\begin{equation}\label{Ineq:unitary}
\iint_{\Sp \times \Sp} \frac{\lrangle{\Phi(\omega)}{\Phi(\omega')}}{|\omega - \omega'|^{\lambda}}\,d\omega\,d \omega'
\le \iint_{\Sp \times \Sp} \frac{ |\Phi(\omega)|\,|\Phi(\omega')| }{|\omega - \omega'|^{\lambda}}\,d\omega\,d \omega' \,.
\end{equation}
By the sharp Hardy-Littlewood-Sobolev inequality~\eqref{nonspinor} applied with $F=|\Phi|$, we obtain
\be{eq:hlsspinor}
\iint_{\Sp \times \Sp} \frac{\lrangle{\Phi(\omega)}{\Phi(\omega')}}{|\omega - \omega'|^{\lambda}}\,d\omega\,d \omega'\le C_p\,\Vert \Phi \Vert_{\mathrm L^p(\Sp)}^2 \,.
\ee
Meanwhile, we have
\begin{equation}
\label{eq:hlsspinoridentity}
\iint_{\Sp \times \Sp} \frac{\lrangle{\Phi(\omega)}{\Phi(\omega')}}{|\omega - \omega'|^{\lambda}}\,d\omega\,d \omega'
= C_p \sum_{\ell=0}^\infty \gamma_\ell\big(\tfrac2p\big)\,\Vert \Phi_\ell\Vert_{\mathrm L^2(\Sp, \C^2)}^2\,.
\end{equation}
Indeed, to see this, we can decompose $\Phi_\ell = \begin{psmallmatrix}\Phi^1_\ell \\ \Phi^2_\ell\end{psmallmatrix}$ and the components $\Phi^1_\ell$ and $\Phi^2_\ell$ are complex-valued spherical harmonics of degree $\ell$. Since
\[
\iint_{\Sp \times \Sp}\kern-3pt \frac{\lrangle{\Phi(\omega)}{\Phi(\omega')}}{|\omega - \omega'|^{\lambda}}\,d\omega\,d \omega' = \sum_{\ell=0}^\infty\,\iint_{\Sp \times \Sp}\kern-3pt \frac{\overline{\Phi_\ell^1 (\omega)}\,\Phi_\ell^1(\omega') }{|\omega - \omega'|^{\lambda}}\,d\omega + \sum_{\ell=0}^\infty\,\iint_{\Sp \times \Sp}\kern-3pt \frac{ \overline{\Phi_\ell^2 (\omega)}\,\Phi_\ell^2(\omega') }{|\omega - \omega'|^{\lambda}}\,d\omega\,,
\]
we can apply the scalar result~\eqref{eq:funkhecke} and obtain the claimed equality~\eqref{eq:hlsspinoridentity}.

By combining~\eqref{eq:hlsspinor} and~\eqref{eq:hlsspinoridentity} we find
\begin{equation}
\label{eq:hlsspinorprimal}
\sum_{\ell=0}^\infty \gamma_\ell\big(\tfrac2p\big)\,\Vert \Phi_\ell\Vert_{\mathrm L^2(\Sp, \C^2)}^2 \le \Vert \Phi\Vert_{\mathrm L^p(\Sp, \C^2)}^2\,,
\end{equation}
and, by duality, with $q=p/(p-1)$, we obtain
\[
\Vert \Phi \Vert_{\mathrm L^q(\Sp, \C^2)}^2 \le \sum_{\ell=0}^\infty \gamma_\ell\big(\tfrac2q\big)\,\Vert \Phi_\ell\Vert_{\mathrm L^2(\Sp, \C^2)}^2\,.
\]
The latter is a rewriting of~\eqref{spinorbeckner}.

We now consider the cases of equality in~\eqref{eq:hlsspinorprimal}. Equality in~\eqref{Ineq:unitary} holds if and only if $\lrangle{\Phi(\omega)}{\Phi(\omega')} = |\Phi(\omega)|\,|\Phi(\omega')|$ for a.e.~$(\omega,\omega')\in\Sp\times\Sp$. Writing
$$
\Phi(\omega) = |\Phi(\omega)|\,U(\omega)\,\Phi_0 \,,
$$
where $\Phi_0 = \begin{psmallmatrix}1\\0\end{psmallmatrix}$ and $U(\omega)$ is a $2\times 2$ unitary matrix, we see that this is equivalent to $\lrangle{U(\omega)\,\Phi_0}{U(\omega')\,\Phi_0}=1$ for a.e.~$(\omega,\omega')\in \mathrm{supp}\,\Phi \times \mathrm{supp}\,\Phi$. This implies that there is a constant spinor $\tilde\Phi$ such that $U(\omega)\,\Phi_0 =\tilde\Phi$ for a.e.~$\omega\in \mathrm{supp}\,\Phi$. (Note that this is the case even if $\mathrm{supp}\,\Phi$ has several connected components.) Thus, $\Phi/|\Phi|$ is a constant spinor on $\{|\Phi|> 0\}$. Equality in~\eqref{eq:hlsspinor} implies that $F=|\Phi|$ is an optimizer in the Hardy-Littlewood-Sobolev inequality~\eqref{nonspinor}. By Lieb's characterization of those~\cite{MR717827} we learn that $F(\omega) = c\,(1-\zeta\cdot\omega)^{-(4-\lambda)/2}$ for some $\zeta\in\R^3$ with $|\zeta|<1$ and some $c>0$. Thus, equality in~\eqref{eq:hlsspinorprimal} holds if and only if $\Phi(\omega) = (1-\zeta\cdot\omega)^{-(4-\lambda)/2}\,\chi_0$ for some $\chi_0\in\C^2$ and some $\zeta$ as before.

By duality (more precisely, by the characterization of cases of equality in H\"older's inequality) we deduce the claimed form of functions for which Inequality~\eqref{spinorbeckner} is saturated.
\end{proof}
\subsubsection*{$\bullet$ An estimate on the decomposition into spherical harmonics.} Inequality~\eqref{spinorbeckner} is phrased in terms of the decomposition of $\Phi$ in terms of spherical harmonics that are eigenfunctions of the operator $L^2$. In contrast, Inequality~\eqref{Ineq:GenPoincare} in Theorem~\ref{thm:BecknerGenSpinor} involves the operator $\sigma\cdot L$ and therefore we now decompose $\Phi$ with respect to the spectrum of the latter operator, that is, with respect to spinor spherical harmonics. We write
\be{Phi-chik}
\Phi = \sum_{k=-\infty}^{-2} \chi_k + \sum_{k=0}^\infty \chi_k\,.
\ee
This decomposition is related to the decomposition in Lemma~\ref{Lem:SpinorialS2interpolation} by $\Phi_\ell = \chi_\ell + \chi_{-(\ell+1)}$ for all $\ell\geq 0$. For a proof of this fact we refer to Appendix~\ref{App5}; see, in particular, Corollary~\ref{decompalt}. Consequently, we can rewrite~\eqref{spinorbeckner} as
\[
\frac{\nrms\Phi q^2-\nrms\Phi 2^2}{q-2}\le \sum_{k=1}^\infty\zeta_k(q)\is{|\chi_k|^2}+ \sum_{k=-\infty}^{-2}\zeta_{-(k+1)}(q)\is{|\chi_k|^2}\,.
\]
Noting that
$$
\is{\lrangle\Phi{\mathcal K_\alpha\,\Phi}} = \sum_{k\neq - 1} k\,(k+1-2\,\alpha) \is{|\chi_k|^2} \,,
$$
we obtain Inequality~\eqref{Ineq:GenPoincare} with
\begin{equation}
\label{eq:defb}
\mathcal B_{\alpha,q} := \min\left\{ \inf_{k\geq 1} \frac{k\,(k+1-2\,\alpha)}{\zeta_k(q)} \,,\, \inf_{k\leq -2} \frac{k\,(k+1-2\,\alpha)}{\zeta_{-(k+1)}(q)} \right\} \,.
\end{equation}
Thus, the proof of the first assertion of Theorem~\ref{thm:BecknerGenSpinor} is reduced to the proof that $\mathcal B_{\alpha,q}>0$ for any $\alpha\in (-1/2,1)$ and $q\in(2,\infty)$. Likewise, the proof of the second assertion is reduced to the proof that, for $q\leq\mathsf q(\alpha)$, we have $\mathcal B_{\alpha,q}=2 \,\mathsf m(\alpha)$ and the infimum in the definition of $\mathcal B_{\alpha,q}$ is attained only at either $k=1$ or $k=-\,2$. We begin with the proof of the latter assertion.
\begin{lemma}\label{lem:ap1} Let $\alpha\in(-1/2,1)$. Then the following inequalities hold:
\begin{enumerate}
\item[(a)] Assuming $2<q<\infty$ for $\alpha\in[0,1)$ and $2<q\le 4-\frac6\alpha$ for $\alpha\in(-1/2,0)$, we have
\begin{equation}
\label{II-a}
\zeta_\ell(q)\le\frac{\ell\,(\ell+1-2\,\alpha)}{2\,(1-\alpha)}
\quad\text{for all}\ \ell\in\N\setminus\{0\}\,.
\end{equation}
\item[(b)] Assuming $2<q<\infty$ when $\alpha\in(-1/2,0]$ and $2< q\le 2+\frac2\alpha$ when $\alpha\in(0,1)$, we have
\begin{equation}
\label{II-b}
\zeta_\ell(q)\le\frac{(\ell+1)\,(\ell+2\,\alpha)}{2\,(1+2\,\alpha)}
\quad\text{for all}\ \ell\in\N\setminus\{0\}\,.
\end{equation}
\end{enumerate}
There is equality in~\eqref{II-a} and~\eqref{II-b} for $\ell=1$ while both are strict for $\ell\geq 3$. For $\ell=2$ there is equality if and only if $q=4-\frac 6\alpha$, $\alpha\in(-1/2,0)$ in the first case and $q=2+\frac2\alpha$, $\alpha\in(0,1)$ in the second case.
\end{lemma}

To prove Lemma~\ref{lem:ap1} we extend a method of~\cite[proof of Theorem~7]{MR4782463} for scalar functions in the special case $\alpha=0$ to the case $\alpha\neq 0$.

\begin{proof}
Throughout the proof the parameter $\alpha\in(-1/2,1)$ is fixed. Let us define
\[
z_\ell^+(\alpha,q):=\frac{2\,(1-\alpha)}{\ell\,(\ell+1-2\,\alpha)}\,\zeta_\ell(q)
\quad\text{and}\quad
z_\ell^-(\alpha,q):=\frac{2\,(1+2\,\alpha)}{(\ell+1)\,(\ell+2\,\alpha)}\,\zeta_\ell(q) \,.
\]
An elementary computation shows that
\begin{equation}
\label{eq:iteration}
z_{\ell+1}^\pm(\alpha,q)=\mathsf a_\ell^\pm(\alpha,q)+\mathsf b_\ell^\pm(\alpha,q)\,z_\ell^\pm(\alpha,q)
\end{equation}
with
\[
\mathsf a_\ell^+(\alpha,q):=\frac{4\,(1-\alpha)}{(\ell+1)\,(\ell+2-2\,\alpha)\,(\ell\,q+2)}\,,\quad
\mathsf b_\ell^+(\alpha,q):=\frac{\ell\,(\ell+1-2\,\alpha)\,\big((\ell+2)\,q-2\big)}{(\ell+1)\,(\ell+2-2\,\alpha)\,(\ell\,q+2)}
\]
and
\[
\mathsf a_\ell^-(\alpha,q):=\frac{4\,(1+2\,\alpha)}{(\ell+2)\,(\ell+1+2\,\alpha)\,(\ell\,q+2)}\,,\quad
\mathsf b_\ell^-(\alpha,q):=\frac{(\ell+1)\,(\ell+2\,\alpha)\,\big((\ell+2)\,q-2\big)}{(\ell+2)\,(\ell+1+2\,\alpha)\,(\ell\,q+2)}\,.
\]
We now make two claims concerning $\mathsf a_\ell^\pm(\alpha,q)+\mathsf b_\ell^\pm(\alpha,q)$. The first claim is that if $2<q<\infty$ for $\alpha\in[0,1)$ and $2<q<4-\frac 6\alpha$ for $\alpha\in(-1/2,0)$, then
\begin{equation}
\label{eq:iterationa2}
\mathsf a_\ell^+(\alpha,q)+\mathsf b_\ell^+(\alpha,q)<1\quad\forall\,\ell\ge1 \,.
\end{equation}
Moreover, if $q=4-\frac 6\alpha$ for $\alpha\in(-1/2,0)$, then the inequality in~\eqref{eq:iterationa2} holds for all $\ell\geq 2$, while it is an equality for $\ell=1$. The second claim is that if $2<q<\infty$ for $\alpha\in(-1/2,0]$ and $2<q<2+\frac2\alpha$ for $\alpha\in(0,1)$, then
\begin{equation}
\label{eq:iterationb2}
\mathsf a_\ell^-(\alpha,q)+\mathsf b_\ell^-(\alpha,q)<1\quad\forall\,\ell\ge1 \,.
\end{equation}
Moreover, if $q=2+\frac2\alpha$ for $\alpha\in(0,1)$, then the inequality in~\eqref{eq:iterationb2} holds for all $\ell\geq 2$, while it is an equality for $\ell=1$.

Before proving these two claims, let us use them to complete the proof of the lemma. We observe that $z_1^\pm(\alpha,q)=1$. By induction, this together with~\eqref{eq:iteration},~\eqref{eq:iterationa2} and~\eqref{eq:iterationb2} implies that $z_\ell^\pm(\alpha,q)<1$ for any $\ell\ge2$. This completes the proof of~\eqref{II-a} and~\eqref{II-b}, even with strict inequality for $\ell\geq 2$, except for the cases $q=4-\frac 6\alpha$, $\alpha\in(-1/2,0)$ for~\eqref{II-a} and $q=2+\frac 2\alpha$, $\alpha\in(-1/2,0)$ for~\eqref{II-b}. In these exceptional cases, the equality in~\eqref{eq:iterationa2} and~\eqref{eq:iterationb2} for $\ell=1$ implies that $z_2^\pm(\alpha,q)=1$. The strict inequalities for $\ell\geq 2$ then allow us to repeat the induction argument and arrive again at~\eqref{II-a} and~\eqref{II-b}, now with strict inequality for $\ell\geq 3$.

Thus, we have reduced the proof of the lemma to the proof of~\eqref{eq:iterationa2} and~\eqref{eq:iterationb2}. We set
$$
F_\ell^\pm(\alpha,q) := \mathsf a_\ell^\pm(\alpha,q)+\mathsf b_\ell^\pm(\alpha,q) \,.
$$
It is convenient to consider these as functions of $q$ in the interval $(-2/\ell,\infty)$. We note that $F_\ell^\pm(\alpha,\cdot)$ is a quotient of affine functions. A computation, plus some straightforward estimates show that $F_\ell^\pm(\alpha,\cdot)$ is an increasing function in the interval $(-2/\ell,\infty)$. Moreover, as $q\to\infty$, it converges to its horizontal asymptote $y_\ell^\pm(\alpha)$, given by
$$
y_\ell^+(\alpha) :=\frac{(\ell+1-2\,\alpha)\,(\ell+2)}{(\ell+1)\,(\ell+2-2\,\alpha)}
\quad\text{and}\quad
y_\ell^-(\alpha) := \frac{(\ell+1)\,(\ell+2\,\alpha)}{(\ell+1+2\,\alpha)\,\ell}
\,.
$$
To complete the proof of~\eqref{eq:iterationa2} and~\eqref{eq:iterationb2}, we now consider these two cases separately.
\\[4pt]
\emph{Proof of Inequality~\eqref{eq:iterationa2}.}
A computation shows that $y^+_\ell(\alpha)\leq 1$ if $\alpha\in[0,1]$. Thus, in this case we have $F_\ell^+(\alpha,q)<y^+_\ell(\alpha)\leq 1$ for all $q>-\frac2\ell$, which immediately implies the assertion.

For $\alpha<0$ we observe that $F_\ell^+(\alpha,q_\ell^+(\alpha))=1$ for $q_\ell^+(\alpha):= 4-2\,\frac{\ell+2}{\alpha}$. Thus, by monotonicity, we have $F_\ell^+(\alpha,q)<1$ for all $-\,\frac2\ell<q<q_\ell^+(\alpha)$. In particular, we have $F_\ell^+(\alpha,q)<1$ for all $\ell\geq 1$ when $0\leq q<4-\frac 6\alpha=\min_{\ell\geq 1} q_\ell^+(\alpha)$. We also see that for $q=4-\frac 6\alpha$, this inequality holds for all $\ell\geq 2$ with equality for $\ell=1$. This proves~\eqref{eq:iterationa2}.
\\[4pt]
\emph{Proof of Inequality~\eqref{eq:iterationb2}.}
A computation shows that $y^-_\ell(\alpha)\leq 1$ if $\alpha\in(-1/2,0]$. Thus, in this case we have $F_\ell^-(\alpha,q)<y^-_\ell(\alpha)\leq 1$ for all $q>-\frac2\ell$, which immediately implies the assertion.

For $\alpha>0$ we observe that $F_\ell^-(\alpha,q_\ell^-(\alpha))=1$ for $q_\ell^-(\alpha):= 2\,\big( \frac\ell\alpha + \frac{1+2\ell}{2+\ell}\big)$. Thus, by monotonicity, we have $F_\ell^-(\alpha,q)<1$ for all $-\,\frac2\ell<q<q_\ell^-(\alpha)$. In particular, we have $F_\ell^-(\alpha,q)<1$ for all $\ell\geq 1$ when $0\leq q<2+\frac 2\alpha=\min_{\ell\geq 1} q_\ell^-(\alpha)$. We also see that for $q=2+\frac2\alpha$, this inequality holds for all $\ell\geq 2$ with equality for $\ell=1$. This proves~\eqref{eq:iterationb2}.
\end{proof}

\begin{remark}\label{ap1rem}
For later purposes we record that clearly $\lim_{\ell\to+\infty}\big(\mathsf a_\ell^\pm(\alpha,q),\mathsf b_\ell^\pm(\alpha,q)\big)=(0,1)$. A more precise analysis shows that $\mathsf a_\ell^\pm(\alpha,q)+\mathsf b_\ell^\pm(\alpha,q)<1$ for all $\ell\in\N$ large enough (depending on $\alpha$ and $q$). Proceeding as in the above prove, we deduce that $z_\ell^\pm(\alpha,q)$ achieves a maximum for some finite $\ell\ge2$. In particular, we have that
$$
\sup_{\ell\geq 1} \frac{2\,(1-\alpha)}{\ell\,(\ell+1-2\,\alpha)}\,\zeta_\ell(q) + \sup_{\ell\geq 1} \frac{2\,(1+2\,\alpha)}{(\ell+1)\,(\ell+2\,\alpha)}\,\zeta_\ell(q) <\infty \,.
$$
\end{remark}

\subsubsection*{$\bullet$ Proof of Theorem~\ref{thm:BecknerGenSpinor}.}
We recall the $\mathcal B_{\alpha,q}$ was defined in~\eqref{eq:defb}. We can rewrite this as
$$
\mathcal B_{\alpha,q} = \min\left\{ \inf_{\ell\geq 1} \frac{\ell(\ell+1-2\,\alpha)}{\zeta_\ell(q)} \,,\, \inf_{\ell\geq 1} \frac{(\ell +1)\,(\ell+2\,\alpha)}{\zeta_\ell(q)} \right\} \,.
$$
It follows from Remark~\ref{ap1rem} that $\mathcal B_{\alpha,q}>0$, which, according to the discussion before Lemma~\ref{lem:ap1}, implies the first assertion of Theorem~\ref{thm:BecknerGenSpinor}.

Moreover, recalling the definition of $\mathsf q(\alpha)$ in~\eqref{q}, we deduce from Lemma~\ref{lem:ap1} that
$$
\mathcal B_{\alpha,q} = \min\Big\{ 2\,(1-\alpha)\,,\, 2\,(1+2\,\alpha)\Big\} = 2 \, \mathsf m(\alpha)\quad\mbox{if}\quad q\leq\mathsf q(\alpha)\,.
$$
This proves the second assertion of Theorem~\ref{thm:BecknerGenSpinor}.

Still assuming $q\leq q(\alpha)$, let us discuss the cases of equality in~\eqref{Ineq:GenPoincare}. Clearly, both sides of the inequality vanish for constant spinors. Conversely, assume that equality holds in the inequality for some spinor $\Phi$ decomposed according to~\eqref{Phi-chik}. By Lemma~\ref{lem:ap1}, we deduce that $\chi_k=0$ for all $k\neq -\,3$, $-\,2$, $0$, $1$, $2$. (The cases $k=-\,3$, $2$ are only relevant when $q=\mathsf q(\alpha)$.)

We note that both components of angular spinors in $\mathcal H_{-3}\oplus\mathcal H_{-2} \oplus\mathcal H_0\oplus \mathcal H_{1}\oplus\mathcal H_{2}$ are polynomials in $\omega_1$, $\omega_2$, $\omega_3$ of degree $\leq 2$. Meanwhile, from Lemma~\ref{Lem:SpinorialS2interpolation} we know that $\Phi(\omega) = (1-\zeta\cdot\omega)^{-2/q}\chi_0$ for some $\zeta\in\R^3$ with $|\zeta|<1$ and $\chi_0\in\C^2$. We deduce that $\zeta=0$ (unless \hbox{$\chi_0=0$}). Thus, $\Phi$ is a constant spinor, as claimed. This completes the proof of Theorem~\ref{thm:BecknerGenSpinor}.
\qed

\subsection{A symmetry range}\label{Sec:symmetryrange}

In this subsection, our goal is to characterize a set of parameters $(\alpha, p)\in(-1/2, 1]\times(2, 6)$ such that $C^\star_{\alpha, p}= C_{\alpha,p}$. We adapt the method of~\cite{MR2905786}, but significant changes have to be implemented in order to deal with spinors.
\begin{theorem}\label{Thm:symmetry1} If one of the following conditions is satisfied,
\begin{subequations}
\begin{align}
&\label{SymZone1}
\alpha\in\big(-\tfrac12,0\big] \quad\mbox{and}\quad 2<p\le2\,\frac{3+4\,(\alpha+1)^2}{3+4\,\alpha^2}\,,\\
&\label{SymZone2}
\alpha\in\big(0,\tfrac12\big)\cup\big(\tfrac12,1\big)\quad\mbox{and}\quad 2<p\le\min\left\{2\,\frac{3+2\,\alpha}{1+2\,\alpha}, \; 2\,\frac{7-10\,\alpha+4\,\alpha^2}{3-6\,\alpha+ 4\,\alpha^2}\right\}\,,
\end{align}
\end{subequations}
then $C_{\alpha,p}=C^\star_{\alpha,p}$ and any optimizer for Inequality~\eqref{logarithmicSCKN} is symmetric of type $\mathcal H_0$.
\end{theorem}
\smallskip\noindent See Fig.~\ref{Fig:SCKNlog-Detail-1} for a plot of Conditions~\eqref{SymZone1} and~\eqref{SymZone2}.

\smallskip\noindent\emph{Proof.} We split the proof in several steps.
\subsubsection*{$\rhd$ First step} The following result was suggested by Keller in~\cite{MR0121101} and proved independently by Lieb and Thirring in~\cite{Lieb-Thirring76}. The underlying one-dimensional interpolation inequality was established by Nagy~\cite{MR4277}. As it is stated below, it is taken from~\cite{MR2905786}. Let us define
\be{CLT}
c_{\rm LT}(\gamma):=\frac{\pi^{-1/2}}{\gamma-1/2}\,\frac{\Gamma(\gamma+1)}{\Gamma(\gamma+1/2)}\(\frac{\gamma-1/2}{\gamma+1/2}\)^{\gamma+1/2}\,.
\ee
\begin{lemma}\label{Lem:Keller} Let $V=V(s)$ be a non-negative real valued potential in $\mathrm L^{\gamma+1/2}(\R)$ for some $\gamma>1/2$. If $-\,\lambda_1(V)$ is the lowest eigenvalue of the Schr\"odinger operator $-\,\frac{d^2}{ds^2}-V$, then
\be{Ineq:Keller}
\lambda_1(V)^\gamma\le c_{\rm LT}(\gamma)\ir{ V^{\gamma +1/2}(s)}\,,
\ee
with equality if and only if
\[
V(s)=B^2\,V_0\big(B\,(s-C)\big)\quad\mbox{and}\quad V_0(s):=\frac{\gamma^2-1/4}{(\cosh s)^2}\quad\forall\,s\in\R\,,
\]
where $B>0$, $C\in\R$ are constants. In that case, $\lambda_1(V_0)=\(\gamma-1/2\)^2$ and the corresponding eigenspace is generated by
\[
u_0(s)=\(\frac{\Gamma(\gamma)}{\sqrt\pi\,\Gamma(\gamma-1/2)}\)^{1/2}\,\big(\cosh s\big)^{1/2-\gamma}\,.
\]
\end{lemma}
\noindent See Appendix~\ref{App00} for a proof. Here $u_0$ is normalized by the condition $\nrmr{u_0}2=1$.
\medskip Let
\[
\mathcal F_\alpha[\phi]:=\nrmc{\partial_s\phi}2^2+\big\|\big(\sigma\cdot L+\tfrac12-\alpha\big)\,\phi\|_{\mathrm L^2(\R\times\Sp)}^2-C_{\alpha,p}\,\nrmc{\phi}p^2\,.
\]
We assume that $\phi$ is an optimizer for~\eqref{logarithmicSCKN}. By homogeneity we may assume without loss of generality that
\be{assumption}
\nrmc{\phi}p^{p-2}=C_{\alpha,p}\le C^\star_{\alpha,p}\,.
\ee
It follows from this normalization that $\phi$ satisfies the Euler--Lagrange equation in the form
\be{EL}
-\,\partial_s^2\,\phi + \big(\sigma\cdot L-\alpha+\tfrac12\big)^2\,\phi = |\phi|^{p-2}\,\phi\quad\text{in}\quad\R\times\Sp\,.
\ee
Multiplying this equation by $\phi$ and integrating, we deduce that
\begin{equation}
\label{eq:phinormalization}
0=\ic{\(|\partial_s\phi|^2+\big|\big(\sigma\cdot L-\alpha+\tfrac12\big)\,\phi\big|^2-|\phi|^p\)}=\mathcal F_\alpha[\phi]\,.
\end{equation}

Similar estimates apply in the symmetric case. Indeed, for $\alpha>-1/2$, symmetric optimizers for~\eqref{logarithmicSCKN} can be rewritten, after a translation, as
\be{phistar}
\phi_\star(s,\omega) = u_\star(s)\,\chi \,,
\ee
where $\chi\in\C^2$ and $u_\star(s)=A\,u_0(B\,s)$ with $B=\frac12\,(p-2)\,\sqrt\lambda$ and $\lambda = \big(\alpha-\tfrac12\big)^2$, that is,
\be{ustar}
u_\star(s)=\(\tfrac{p\,\lambda}2\)^\frac1{p-2}\(\cosh\(\tfrac12\,(p-2)\,\sqrt\lambda\,s\)\)^{-\frac2{p-2}}\quad\forall\,s\in\R\,.
\ee
Without loss of generality, we shall assume that $|\chi|=1$.
See Appendix~\ref{App00} for details. As above, let us choose $A$ such that
\begin{equation}
\label{eq:normalizationsymm}
\nrmr{u_\star}p^{p-2}=C^\star_{\alpha,p}\,.
\end{equation}
As a consequence of this and~\eqref{assumption}, we obtain
\be{phi-phistar}
\ic{|\phi|^p}\le\ic{|\phi_\star|^p}\,.
\ee
For the same reasons as before, $\phi_\star$ solves~\eqref{EL} and
\be{eq:phistarnormalization}
0=\ic{\(|\partial_s\phi_\star|^2+\big|\big(\sigma\cdot L-\alpha+\tfrac12\big)\,\phi_\star\big|^2-|\phi_\star|^p\)}\,.
\ee

After these preliminaries, we establish a lower bound for $\mathcal F_\alpha[\phi]$. Let
\be{gamma:p}
\gamma=\frac12\,\frac{p+2}{p-2}>\frac12\,.
\ee
For a.e.~\hbox{$\omega\in\Sp$}, we apply Lemma~\ref{Lem:Keller} with $V(s)=|\phi(s,\omega)|^{p-2}$ and infer that the lowest eigenvalue $-\,\lambda_1(V)$ of the operator $\,-\frac{d^2}{ds^2}- V\,$ acting in $\mathrm L^2(\R\times\Sp)$ is bounded from below according to
$$
-\,\lambda_1(V)\ge-\,c_{\rm LT}(\gamma)^{1/\gamma}\(\ir{|\phi(s,\omega)|^p}\)^{1/\gamma}\,.
$$
The same estimate applies to the lowest eigenvalue of the corresponding operator acting in $\mathrm L^2(\R\times\Sp,\C^2)$, since the operator acts trivially on the spin. By the variational characterization of the first eigenvalue, we deduce that for a.e.~$\omega\in\Sp$ we have
\[
\ir{\(|\partial_s\phi(s,\omega)|^2-|\phi(s,\omega)|^p\)}\ge-\,c_{\rm LT}(\gamma)^{1/\gamma}\(\ir{|\phi(s,\omega)|^p}\)^{1/\gamma}|u(\omega)|^2\,,
\]
where we set
\[
u(\omega):=\sqrt{\ir{|\phi(s,\omega)|^2}}\,.
\]

Integrating this inequality with respect to $\omega\in\Sp$ we obtain
\[
\mathcal F_\alpha[\phi]\ge -\,c_{\rm LT}(\gamma)^{1/\gamma}\is{\(\ir{|\phi(s,\omega)|^p}\)^{1/\gamma}|u(\omega)|^2}+\ic{\left|\(\sigma\cdot L-\alpha+\tfrac12\)\phi\right|^2}\,.
\]

\subsubsection*{$\rhd$ Second step} Applying Theorem~\ref{thm:BecknerGenSpinor} to $\phi(s,\cdot)$ for almost every $s\in\R$ and integrating with respect to $s$ we find that, for each $q\in\big(2,\infty\big)$,
\begin{align*}
&\ic{\left|\(\sigma\cdot L-\alpha+\tfrac12\)\phi\right|^2} = \ic{\lrangle\phi{\mathcal K_\alpha\,\phi}} +\(\alpha -\tfrac12\)^2 \ic{|\phi|^2}
\\
&\ge\frac{\mathcal B_{\alpha,q}}{q-2}\int_\R \(\(\is{|\phi(s,\omega) |^q}\)^{2/q}-\is{|\phi(s,\omega)|^2}\) \,ds +\(\alpha -\tfrac12\)^2 \ic{|\phi|^2}\,.
\end{align*}
By Minkowski's inequality (see for instance~\cite[Section~2.4, p.~47]{MR1817225}),
\[
\(\is {u(\omega)^q}\)^{2/q} = \(\is {\({\ir{|\phi(s,\omega)|^2}}\)^{q/2}}\)^{2/q} \le \int_\R \(\is{|\phi(s,\omega) |^q}\)^{2/q}\,ds\,.
\]
By combining this with the inequality from Step 1, we obtain for all $q\in\big(2,\mathsf q(\alpha)\big)$,
\begin{multline*}
\mathcal F_\alpha[\phi]\ge -\,c_{\rm LT}(\gamma)^{1/\gamma}\is{\(\ir{|\phi(s,\omega)|^p}\)^{1/\gamma}|u(\omega)|^2}
\\
+\frac{\mathcal B_{\alpha,q}}{q-2} \(\(\is{|u(\omega)|^q}\)^{2/q}-\is{|u(\omega)|^2}\) +\(\alpha -\tfrac12\)^2 \is{|u(\omega)|^2}\,.
\end{multline*}

\subsubsection*{$\rhd$ Third step} Assuming $\gamma >1$, we apply H\"older's inequality,
\[
\is{\(\ir{|\phi(s,\omega)|^p}\)^\frac1\gamma|u(\omega)|^2}\le\(\ic{|\phi(s,\omega)|^p}\)^\frac1\gamma\(\is{|u(\omega)|^\frac{2\,\gamma}{\gamma-1}}\)^\frac{\gamma-1}\gamma\,.
\]
Thus, setting
\[
D:=c_{\rm LT}(\gamma)^{1/\gamma}\(\ic{|\phi(s,\omega)|^p}\)^\frac1\gamma\,,
\]
we obtain
\begin{multline*}
\mathcal F_\alpha[\phi]\ge - D\(\is{|u(\omega)|^{\frac{2\,\gamma}{\gamma-1}}}\)^\frac{\gamma-1}{\gamma}+\(\alpha -\tfrac12\)^2 \is{|u(\omega)|^2}
\\
+\frac{\mathcal B_{\alpha,q}}{q-2} \(\(\is{|u(\omega)|^q}\)^{2/q}-\is{|u(\omega)|^2}\)\,.
\end{multline*}
At this point we choose the parameter $q$ such that
\be{pqgamma}
q=\frac{2\,\gamma}{\gamma-1}=2\,\frac{p+2}{6-p} \,.
\ee
This choice is consistent with the requirement $\gamma\in(1,+\infty)$ if and only if $p\in(2,6)$. We find that
\[
\mathcal F_\alpha[\phi]\ge \(\frac{\mathcal B_{\alpha,q}}{q-2}-D\) \(\is{|u|^q}\)^\frac2q+\(\(\alpha-\tfrac12\)^2-\frac{\mathcal B_{\alpha,q}}{q-2}\)\is{|u|^2} =: \mathcal E[u]\,.
\]

\subsubsection*{$\rhd$ Fourth step} By~\eqref{phi-phistar}, we know that
\be{D-estimate}
D= c_{\rm LT}(\gamma)^\frac1\gamma\(\ic{|\phi(s,\omega)|^p}\)^\frac1\gamma\le c_{\rm LT}(\gamma)^\frac1\gamma\(\ic{|\phi_\star|^p}\)^\frac1\gamma=\(\alpha-\tfrac12\)^2\,.
\ee
The last equality follows from the explicit value of $c_{\rm LT}(\gamma)$ in~\eqref{CLT} and the explicit value of~\eqref{eq:normalizationsymm} given in Proposition~\ref{symmetricproblem}. More conceptually, one can also note that $-\,\big(\alpha-\tfrac12\big)^2$ is the smallest eigenvalue of the Schr\"odinger operator $-\,\partial_s^2 - |\phi_\star|^{p-2}$ and that the potential $|\phi_\star|^{p-2}$ saturates the inequality in Lemma~\ref{Lem:Keller} for the given value of $\gamma$.

We now assume, in addition, that we have
\begin{equation}
\label{eq:symm2proofass}
\(\alpha-\tfrac12\)^2 \leq \frac{\mathcal B_{\alpha,q}}{q-2} \,.
\end{equation}
Then, from~\eqref{D-estimate} we deduce that
\[
D\le \frac{\mathcal B_{\alpha,q}}{q-2}\,.
\]

To proceed, we use the fact that $d\omega$ is a probability measure, and by H\"older's inequality, we get
\[
\(\is{|u|^q}\)^{\frac2q}\ge\is{|u|^2}\,.
\]
Thus,
\[
\(\frac{\mathcal B_{\alpha,q}}{q-2}-D\)\(\is{|u|^q(\omega)}\)^{\frac2q}\ge\(\frac{\mathcal B_{\alpha,q}}{q-2}-D\)\is{|u|^2}\,,
\]
and therefore
\[
\mathcal E[u]\ge\(\(\alpha-\tfrac12\)^2-D\)\is{|u|^2}\,.
\]
Recalling~\eqref{eq:phinormalization} and using~\eqref{D-estimate} again, we obtain the chain of inequalities
\[
0=\mathcal F_\alpha[\phi]\ge\mathcal E[u]\ge0\,.
\]
This requires that equality must hold in each step. In particular, equality in~\eqref{D-estimate} gives $D=\(\alpha-\tfrac12\)^2$. Since equality in~\eqref{D-estimate} came from the inequality $C_{\alpha,p}\leq C^\star_{\alpha,p}$ in~\eqref{assumption}, we deduce that $C_{\alpha,p} = C^\star_{\alpha,p}$ under Assumption~\eqref{eq:symm2proofass}.

To complete the proof of the first part of Theorem~\ref{Thm:symmetry1}, we will now show that Assumptions~\eqref{SymZone1} and~\eqref{SymZone2} imply Assumption~\eqref{eq:symm2proofass}.

We first assume that $\alpha\in(-1/2,0]$ and consequently that $2<p\leq 2\,\big(3+4\,(\alpha+1)^2\big)/\big(3+4\,\alpha^2\big)$ by~\eqref{SymZone1}. Note that this implies $p<6$, so $q$ is well defined by~\eqref{pqgamma}. Moreover, the bounds on $p$ are equivalent to $2<q\leq 2\,\big(3+4\,(\alpha+1)^2+3+4\,\alpha^2\big)/\big(3\,(3+4\,\alpha^2)-3-4\,(\alpha+1)^2\big)$. One checks that $2\,\big(3+4(\alpha+1)^2+3+4\,\alpha^2\big)/\big(3\,(3+4\,\alpha^2)-3-4\,(\alpha+1)^2\big)<4-6/\alpha=\mathsf q(\alpha)$, so we may assume that $\mathcal B_{\alpha,q}=2\,\mathsf m(\alpha)$. By~\eqref{pqgamma}, Assumption~\eqref{eq:symm2proofass} is equivalent to
\be{betastar}
\(\alpha-\tfrac12\)^2\le2\,\frac{\mathsf m(\alpha)}{q-2}=\tfrac12\,\mathsf m(\alpha)\,\frac{6-p}{p-2}\,,
\ee
which in turn is equivalent to Assumption~\eqref{SymZone1}. See Fig.~\ref{Fig:SCKNlog-Detail-1}.

Now assume that $\alpha\in(0,1/2)\cup(1/2,1)$ and consequently that $2<p\leq\min\big\{ (6+4\,\alpha)/(1+2\,\alpha), 2\,\big(7 - 10\,\alpha + 4\,\alpha^2\big)/\big(3-6\,\alpha+4\,\alpha^2\big)\big\}$ by~\eqref{SymZone2}. Again this implies $p<6$, so $q$ is well defined by~\eqref{pqgamma}. A computation shows that the assumption $p\leq (6+4\,\alpha)/(1+2\,\alpha)$ is equivalent to the assumption $q\leq\mathsf q(\alpha)$, so we may again assume that $\mathcal B_{\alpha,q}=2\,\mathsf m(\alpha)$. The resulting form~\eqref{betastar} of~\eqref{eq:symm2proofass} turns out to be equivalent to the assumption $p\leq 2\,\big(7 - 10\,\alpha + 4\,\alpha^2\big)/\big(3-6\,\alpha+4\,\alpha^2\big)$. Thus, we have shown that Assumption~\eqref{SymZone2} implies Assumption~\eqref{eq:symm2proofass}. See Fig.~\ref{Fig:SCKNlog-Detail-1}. This ends the first part of the proof of Theorem~\ref{Thm:symmetry1}.

\subsubsection*{$\rhd$ Fifth step} Let us prove the stronger result that under the assumptions of Theorem~\ref{Thm:symmetry1} every minimizer is symmetric. In Step~2 we applied Theorem~\ref{thm:BecknerGenSpinor}. The equality statement there implies that $\phi(s,\omega)$ must be a constant spinor for almost all $s$. So, for a.e.~$s\in\R$ there is a unitary $2\times 2$-matrix $U(s)$ and a non-negative number $f(s)$ such that $\phi(s,\cdot) = f(s)\,U(s)\,\chi_0$. In Step 1, we applied Lemma~\ref{Lem:Keller}. More precisely, we applied a corollary of this lemma for an operator acting on $\C^2$-valued functions. This extension from $\C$ to $\C^2$-valued functions implies that there is a unitary $2\times 2$-matrix $U_0$ such that $U(s) = U_0$ for a.e.~$s\in\R$.

The equality statement of Lemma~\ref{Lem:Keller} now implies that there are $B>0$ and $C\in\R$ such that
$$
f(s)^{p-2} = |\phi(s,\omega)|^{p-2} = B^2\,V_0\big(B(s-C)\big)
\quad\text{and}\quad
f(s) = A\,u_0\big(B(s-C)\big)
$$
for some nonnegative constant $A$. Here $V_0$ and $u_0$ are defined in Lemma~\ref{Lem:Keller} and $\gamma$ is related to $p$ by~\eqref{pqgamma}. By translation invariance, we may assume $C=0$. By inserting $\phi=f\,U_0\,\chi_0$ into the Euler--Lagrange equation, we find that necessarily
$$
B = \tfrac12\,(p-2)\left|\tfrac12 - \alpha\right|\,.
$$
Thus, we have shown that $f=u_\star$ given by~\eqref{phistar}. This proves the uniqueness of minimizers up to translations with respect to $s$ and applications of a unitary transformation in $\C^2$, while $A$ is determined by~\eqref{eq:phistarnormalization}.
\qed

\begin{figure}
\floatbox[{\capbeside\thisfloatsetup{capbesideposition={right,center},capbesidewidth=10cm}}]{figure}[\FBwidth]
{\caption{\small For any $\alpha\in[-1/2,1)$, the range of symmetry in Theorem~\ref{Thm:symmetry1} is represented in \greenbw. The dotted curve corresponds to the condition $q=\mathsf q(\alpha)$, rewritten in terms of $p$ using~\eqref{pqgamma}. The plain curves correspond to the condition that $\(\alpha-1/2\)^2=\tfrac12\,\mathsf m\,(6-p)/(p-2)$
with either $\mathsf m=1-\alpha$ or $\mathsf m=1+2\,\alpha$ and determine~\eqref{SymZone1}-\eqref{SymZone2}: symmetry occurs below these curves.
}\label{Fig:SCKNlog-Detail-1}}
{\includegraphics[width=4.5cm]{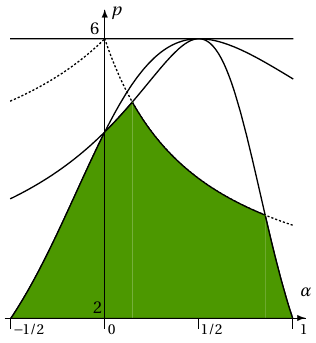}}
\end{figure}

\subsection{Extension}
The range of symmetry in Theorem~\ref{Thm:symmetry1} is limited to two subintervals of $(-1/2,1)\ni\alpha$. Using the transformation $\alpha\mapsto-\,(\alpha+1)$ as in~\eqref{logarithmicalt}, we can extend the result to subintervals of $(-2,-1/2)$ and prove the following result.
\begin{corollary}\label{Cor:symmetry2} If one of the following conditions is satisfied,
\begin{align*}
&-\,1\le\alpha<-\,\tfrac12\quad\mbox{and}\quad2< p\le2\,\frac{3+4\,\alpha^2}{3+4\,(\alpha+1)^2}\,,\\
&\alpha \in \big(-2,-\tfrac32\big)\cup\big(-\tfrac32,-1\big)\quad\mbox{and}\quad2< p\le \min\left\{\frac{4\,\alpha-2}{2\,\alpha+1},\; 2\,\frac{21+18\alpha+4\,\alpha^2}{13+14\,\alpha+4\,\alpha^2} \right\}\,,
\end{align*}
then $C_{\alpha,p}=C^\star_{\alpha,p}$ and any optimizer for Inequality~\eqref{logarithmicSCKN} is symmetric of type $\mathcal H_{-2}$.
\end{corollary}

\section{Instability of symmetric solutions and symmetry breaking}\label{Sec:symmetrybreaking}

The aim of this section is to find regions in the set $(\R\setminus\Lambda)\times(2,6)\ni(\alpha,p)$ for which there is symmetry breaking, that is, $C_{\alpha,p}<C^\star_{\alpha,p}$. We argue as follows. If $\phi_\star\in\mathrm H^1(\Sp,\C^2)$ is a symmetric optimizer for~\eqref{logarithmicSCKN} restricted to symmetric functions, normalized such that $\nrmc{\phi_\star}p^{p-2}=C^\star_{\alpha,p}$, let us consider the quadratic form
\[
\mathcal Q[\varphi]:=\lim_{\varepsilon\to0}\frac1{\varepsilon^2}\(\mathcal F[\phi_\star+\varepsilon\,\varphi]-\mathcal F[\phi_\star]\)
\]
where $\mathcal F[\phi]:=\nrmc{\partial_s\phi}2^2+\big\|\big(\sigma\cdot L+\tfrac12-\alpha\big)\,\phi\|_{\mathrm L^2(\R\times\Sp)}^2-C^\star_{\alpha,p}\,\nrmc{\phi}p^2$. If we can find a function $\varphi\in\mathrm H^1(\Sp,\C^2)$ such that $\mathcal Q[\varphi]<0$, then $\mathcal F[\phi_\star+\varepsilon\,\varphi]$ is also negative if we take $\varepsilon>0$ small enough, and then we know that $C_{\alpha,p}<C^\star_{\alpha,p}$. We shall say that $\phi_\star$ is \emph{linearly unstable} if we can find a function $\varphi$ such that $\mathcal Q[\varphi]<0$. 

Let us start by computing $\mathcal Q$.
\begin{lemma} If $(\alpha,p)\in(\R\setminus\Lambda)\times(2,6)$, then
\begin{multline*}
\mathcal Q[\varphi] =\nrmc{\partial_s\varphi}2^2+\big\|\big(\sigma\cdot L+\tfrac12-\alpha\big)\,\varphi\|_{\mathrm L^2(\R\times\Sp)}^2\\
\hspace*{3cm}- \ic{|\phi_\star|^{p-2}\,|\varphi|^2}-(p-2) \ic{|\phi_\star|^{p-4}\,\big|\mathrm{Re}\lrangle{\phi_\star}\varphi\big|^2}\\
+(p-2)\; \frac{\(\ic{|\phi_\star|^{p-2}\,\mathrm{Re}\lrangle{\phi_\star}\varphi}\)^2}{\nrmc{\phi_\star}p^p}\,.
\end{multline*}
\end{lemma}
\begin{proof} The proof is elementary but requires some care. Using
\[
|\phi_\star+\varepsilon\,\varphi|^2=|\phi_\star|^2+2\,\varepsilon\,\mathrm{Re}\lrangle\varphi{\phi_\star}+\varepsilon^2\,|\varphi|^2\,,
\]
we perform a Taylor expansion at order two in $\varepsilon$ and obtain
\begin{align*}
|\phi_\star+\varepsilon\,\varphi|^p&=|\phi_\star|^p\(1+2\,\varepsilon\,\frac{\mathrm{Re}\lrangle\varphi{\phi_\star}}{|\phi_\star|^2}+\varepsilon^2\,\frac{|\varphi|^2}{|\phi_\star|^2}\)^\frac p2\\
&=|\phi_\star|^p+p\,\varepsilon\,|\phi_\star|^{p-2}\,\mathrm{Re}\lrangle\varphi{\phi_\star}+\frac p2\,\varepsilon^2\,|\phi_\star|^{p-2}\,|\varphi|^2\\
&\hspace*{4cm}+\frac12\,p\,(p-2)\,\varepsilon^2\,|\phi_\star|^{p-4}\(\mathrm{Re}\lrangle\varphi{\phi_\star}\)^2+o\(\varepsilon^2\)\,.
\end{align*}
After integration, we find that
\begin{multline*}
\frac{\nrmc{\phi_\star+\varepsilon\,\varphi}p^2}{\nrmc{\phi_\star}p^2}-1-2\,\varepsilon\,\frac{\ic{|\phi_\star|^{p-2}\,\mathrm{Re}\lrangle\varphi{\phi_\star}}}{\nrmc{\phi_\star}p^{p-2}}\\
=\varepsilon^2\,\frac{\ic{|\phi_\star|^{p-2}\,|\varphi|^2}}{\nrmc{\phi_\star}p^p}+(p-2)\,\varepsilon^2\,\frac{\ic{|\phi_\star|^{p-4}\(\mathrm{Re}\lrangle\varphi{\phi_\star}\)^2}}{\nrmc{\phi_\star}p^p}\\
-\,(p-2)\,\varepsilon^2\,\frac{\(\ic{|\phi_\star|^{p-2}\,\mathrm{Re}\lrangle\varphi{\phi_\star}}\)^2}{\nrmc{\phi_\star}p^{2\,p}}+o\(\varepsilon^2\)\,.
\end{multline*}
Recalling the normalization $\nrmc{\phi_\star}p^{p-2}=C^\star_{\alpha,p}$ completes the proof. \end{proof}

\medskip Let us assume that $\alpha>-1/2$ so that $\phi_\star(s,\omega)=u_\star(s)\,\chi_0$ with a constant spinor $\chi_0\in\C^2$ where $u_\star$ is defined by~\eqref{ustar}. With no loss of generality, we choose $\chi_0=\begin{psmallmatrix}1\\0\end{psmallmatrix}$. Also, after a translation, we may assume that the function $u_\star$ is given by~\eqref{ustar}. An \emph{ansatz} determines a range of the parameters $(\alpha,p)$ for which $\mathcal Q$ takes negative values. Let us consider the set of spinors~$\mathcal S_k^m$ such that
\[
\varphi(s,\omega):=w(s)\,\chi_k^m(\omega)\quad\forall\,(s,\omega)\in\R\times\Sp\,,
\]
where $w\in\mathrm H^1(\R)$ is real valued and such that $\nrmr w2=1$, $k\in\Z\setminus\{-1,0\}$ and $m\in M_k$ where $M_k:=\big\{ k+3/2, k+1/2,\dots, -k-3/2\big\}$ if $\,k\le -2$, and $M_k:=\big\{-k-1/2, -k+1/2,\dots, k+1/2\big\}$ if $k\ge 1$. See Lemma~\ref{basis} in Appendix~\ref{App5} for the definition of $\chi_k^m(\omega)$. With this notation,
\[
\mathcal Q[\varphi]=\ir{\(|w'(s)|^2+\(k-\alpha+\tfrac12\)^2\,|w(s)|^2-\frac{A_k^m\,|w(s)|^2}{\big(\cosh(B\,s)\big)^2}\)}\quad\forall\,\varphi\in\mathcal S_k^m
\]
because $\int_\Sp \lrangle{\chi_0}{\chi_k^m}\,d\omega=0$. Here $A=\frac p2\,|\alpha-1/2|^2$, $B=\frac12\,(p-2)\,|\alpha-1/2|$,
\[
A_k^m:=\(1+(p-2)\,\delta_k^m\)A\quad\mbox{and}\quad\delta_k^m:=\int_\Sp\big|\mathrm{Re}\lrangle{\chi_0}{\chi_k^m}\big|^2\,d\omega\,.
\]
For any $k\in\Z\setminus\{-1,0\}$ and $m\in M_k$, minimizing $\mathcal Q[\varphi]$ with respect to $\varphi\in\mathcal S_k^m$ determines $\varphi_k^m(s,\omega)=w_k^m(s)\,\chi_k^m(\omega)$ where $w_k^m>0$ is uniquely defined according to Lemma~\ref{Lem:Keller}. Taking into account the scaling $s\mapsto B\,s$, we obtain $\mathcal Q\big[\varphi_k^m\big]=(k-\alpha+1/2)^2 - B^2\,(\gamma-1/2)^2$ with~$\gamma$ such that $\gamma^2-1/4=A_k^m/B^2$.
\begin{lemma} Assume that $(\alpha,p)\in(\R\setminus\Lambda)\times(2,6)$. With the above notation, we obtain
\[
\min_{\varphi\in\mathcal S_k^m}\mathcal Q[\varphi]=\mathcal Q\big[\varphi_k^m\big]=\(k-\alpha+\tfrac12\)^2-\frac14\(\sqrt{4\,A_k^m+B^2}-B\)^2=:\frac14\,\mathsf q[k,m;\alpha,p]\,.
\]
\end{lemma}
The next step is to compute these coefficients.
\begin{lemma} For any $k\in\Z\setminus\{-1,0\}$, we have
\[
\delta_k^{1/2}=\frac{k+1}{2\,k+1}\quad\mbox{and}\quad\delta_k^m=\frac 12\,\frac{k+m+1/2}{2\,k+1}\quad\mbox{if}\quad m\neq\frac12\,.
\]
\end{lemma}
\begin{proof} Using the explicit formulas for $\chi_k^m$ in Appendix~\ref{App5}, it follows that
\begin{align*}
&\lrangle{\chi_0}{\chi_k^m}=\sqrt{4\,\pi\,\frac{k+m+1/2}{2\,k+1}}\,Y_\ell^{m-1/2}
\quad\text{where}\ \ell = k \ \text{if}\ k\geq 0 \,,\quad \ell = -k-1 \ \text{if}\ k\leq -2\,.
\end{align*}
We note that, in view of~\eqref{eq:harmonicsconj},
\begin{align*}
\big(\Re Y^{m-1/2}_\ell\big)^2 & = \frac14 \left( \big(Y^{m-1/2}_\ell\big)^2 + 2\,\overline{Y^{m-1/2}_\ell} \,Y^{m-1/2}_\ell + \(\overline{Y^{m-1/2}_\ell}\)^2 \) \\
& = \frac12 \left( (-1)^{m-1/2} \Re \left( \overline{Y^{-m+1/2}_\ell} \,Y^{m-1/2}_\ell \right) + \big|Y^{m-1/2}_\ell\big|^2 \right).
\end{align*}
Using the orthonormality of the spherical harmonics, we obtain the claimed result.
\end{proof}

With these preliminaries in hand, we are now ready to establish regions of symmetry breaking, which are values of $(\alpha,p)$ for which $\mathsf q[k,m;\alpha,p]<0$.

\noindent$\rhd$ If $\alpha\in(-1/2,0)$, then
\[
\mathsf q[1,1/2;\alpha,p]=(3-2\,\alpha)^2-\frac{(1-2\,\alpha)^2}{48}\,\(\sqrt{19\,p^2-20\,p+12}-\sqrt3\,(p-2)\)^2
\]
takes negative values if and only if (see Fig.~\ref{Fig:SCKNlog-Detail-2})
\[
p> p_1(\alpha):=\frac{5-4\,\alpha+\sqrt{8\,\alpha\,(2\,\alpha-11)+97}}{2\,(1-2\,\alpha)}\quad\mbox{and}\quad\alpha<\frac{1-\sqrt3}4\approx-\,0.183013\,.
\]

\noindent$\rhd$ If $\alpha\in(-1/2,0)$, then
\[
\mathsf q[-2,1/2;\alpha,p]=(3+2\,\alpha)^2-\frac{(1-2\,\alpha)^2}{48}\,\(\sqrt{11\,p^2-4\,p+12}-\sqrt3\,(p-2)\)^2
\]
takes negative values if and only if (see Fig.~\ref{Fig:SCKNlog-Detail-2})
\[
p> p_2(\alpha):=2\,\frac{2\,(1+\alpha)+\sqrt{16\,\alpha^2+32\,\alpha+13}}{1-2\,\alpha}\quad\mbox{and}\quad\alpha<\frac{1-\sqrt2}2\approx-\,0.207107\,.
\]

\noindent$\rhd$ If $\alpha\in(1/2,+\infty)$ and $\alpha\not\in\Lambda$, then
\[
\mathsf q[1,1/2;\alpha,p]=(3-2\,\alpha)^2-\frac{(1-2\,\alpha)^2}{48}\,\(\sqrt{19\,p^2-20\,p+12}-\sqrt3\,(p-2)\)^2
\]
takes negative values if and only if
\[
p> p_5(\alpha):=\frac{2-\alpha+\sqrt{25\,\alpha^2-64\,\alpha+40}}{2\,\alpha-1}\quad\mbox{and}\quad\alpha>\frac{3+\sqrt3}6\approx0.788675\,.
\]

\begin{remark}\label{Rem:destabilization}
Note that we have obtained symmetry breaking by destabilization of the symmetric spinor~$\phi_\star$ along single channel functions $\varphi\in\mathcal S_k^m$. For such functions, it is easily seen that the above choices of $k$ and $m$ are the best possible ones. It is not unlikely that by mixing several channels one might increase the region of instability\footnote{Added in proof : Recent work of M.J.~Esteban and R.~Frank suggests that this is the case.}. Finding the optimal region is an open problem.
\end{remark}

\begin{theorem}\label{Thm:symbreak} Let $\alpha\in\R\setminus\Lambda$ and $p\in(2,6)$. Symmetric optimizers are linearly unstable and $C_{\alpha,p}<C^\star_{\alpha,p}$ if one of the following conditions is satisfied:
\begin{enumerate}
\item[\rm (i)] $\alpha\in\(-1/2,\big(1-\sqrt3\big)/4\)$ and $p>\min\big\{p_1(\alpha),p_2(\alpha)\big\}$,
\item[\rm (ii)] $\alpha\in\(\big(3+\sqrt3\big)/6,1\)$ and $p>p_5(\alpha)$,
\item[\rm (iii)] $\alpha>1$.
\end{enumerate}
Moreover, $C_{\alpha,p}<C^\star_{\alpha,p}$ if $\alpha<-1/2$ and $\big(-(1+\alpha),p\big)$ is in an unstable region {\rm (i)}, {\rm (ii)} or {\rm (iii)}.
\end{theorem}

\begin{figure}
\floatbox[{\capbeside\thisfloatsetup{capbesideposition={right,center},capbesidewidth=10cm}}]{figure}[\FBwidth]
{\caption{\small In the range $\alpha\in[-1/2,0)$, a range of linear instability given either by the condition $p>p_1(\alpha)$, or by the condition $p>p_2(\alpha)$ is represented in \redbw. Altogether, we learn from~Theorem~\ref{Thm:symbreak} that optimal spinors are not symmetric if $p$ is larger than $\min\big\{p_1(\alpha),p_2(\alpha)\big\}$.
}\label{Fig:SCKNlog-Detail-2}}
{\includegraphics[width=5cm]{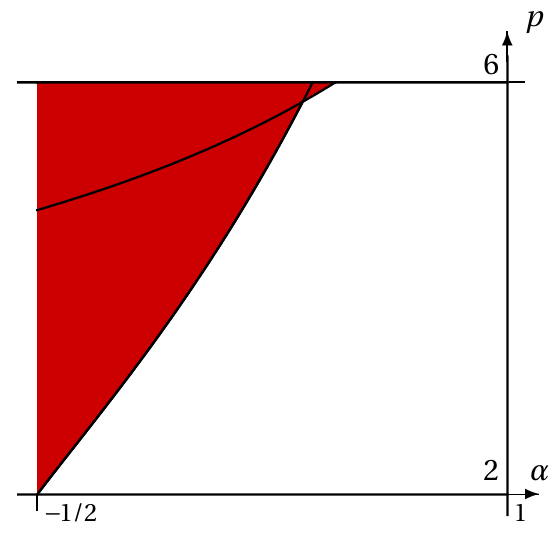}}
\end{figure}
\smallskip\noindent See Fig.~\ref{Fig:SCKNlog-Detail-2} for a plot of $p_1$ and $p_2$ in Case (i). The functions $\alpha\mapsto p_j(\alpha)$ with \hbox{$j=1$, $2$, $5$} are strictly monotone on the intervals considered in Theorem~\ref{Thm:symbreak}, and their inverses $p\mapsto\alpha_j(p)$ are the functions which appear in Theorem~\ref{Thm:main}.

\section{Proof of the main results}\label{Sec:proofmain}

\begin{proof}[Proof of Theorem~\ref{Thm:main}] The existence of optimizers for~\eqref{logarithmicSCKN} for $p\in(2,6)$ and $\alpha\not\in\Lambda$ is proved in Theorem~\ref{Thm:Existence}. The symmetry results for the optimizers are collected from Corollary~\ref{Cor:green}, Theorem~\ref{Thm:symmetry1} and Corollary~\ref{Cor:symmetry2}. Symmetry breaking follows from the linear unstability of Theorem~\ref{Thm:symbreak}. \end{proof}

\begin{proof}[Proof of Theorem~\ref{Cor:main}]
We recall the equivalence between~\eqref{SCKN} and~\eqref{logarithmicSCKN} and define $\overline\alpha_\star$ and $\overline\alpha^\star$ as in Remark~\ref{translation}. The argument of $\alpha_\star$ and $\alpha^\star$ in these formulas is motivated by the formula~\eqref{param} for $p$ in terms of $\alpha$ and $\beta$. The assertions are therefore an immediate consequence of the first part of Theorem~\ref{Thm:main}.
\end{proof}

\appendix\section{The best constant among symmetric functions}\label{App00}

Our goal in this appendix is to compute the optimal constants $\mathcal C_{\alpha,\beta}^\star$ and $C_{\alpha,p}^\star$ in Inequalities~\eqref{SCKN} and~\eqref{logarithmicSCKN} when restricted to symmetric functions.

\medskip We recall that $K_{\lambda,p}$ denotes the optimal constant in Inequality~\eqref{OneD}. The following lemma essentially goes back to Nagy~\cite{MR4277}.
\begin{lemma}\label{nagy} For any $\lambda>0$ and $p\in(2,+\infty)$, the optimal constant in Inequality~\eqref{OneD} is
$$
K_{\lambda,p} = \lambda^{\frac12+\frac1p}\,K_{1,p}\quad\mbox{with}\quad K_{1,p}=\frac p2\(\frac{8\,\sqrt\pi\,\Gamma\big(\tfrac2{p-2}\big)}{(p^2-4)\,\Gamma\big(\tfrac12+\tfrac2{p-2}\big)}\)^{1-\frac2p}\,.
$$
The optimizers are, up to a translation in $s$ and multiplication by a constant, given by~\eqref{ustar}. \end{lemma}
\begin{proof}[Proof of Lemma~\ref{nagy}]
The proof is well known, but for the sake of completeness we sketch its major steps. The constant $K_{\lambda,p}$ is computed from $K_{1,p}$ by a simple scaling argument, so we can assume that $\lambda=1$.

It is easy to see that $K_{1,p}$ is positive and that there is an optimizer $v\in\mathrm H^1(\R)$ for the inequality, which is a nonnegative function. By homogeneity we may assume that $\nrmr vp^{p-2}=K_{1,p}$ so that $v$ satisfies
\[
-\,v''+v=v^{p-1}\quad\text{on}\quad\R\,.
\]
Up to translation, $v(s)=(p/2)^{1/(p-2)}\,\big(\cosh\big((p-2)\,s/2\big)\big)^{-2/(p-2)}$ is the unique positive solution in $\mathrm H^1(\R)$. This implies the claimed form of optimizers and of the optimal constant. We refer to~\cite{MR3263963} for further details.
\end{proof}

Let us continue with~\eqref{logarithmicSCKN} restricted to symmetric spinors as defined in Definition~\ref{symmetric}. In the symmetric problem we may allow for arbitrarily large exponents $p$ and recall that $\mathcal C_{\alpha,\beta}^\star=(4\,\pi)^{-(2-p)/p}\,C_{\alpha,p}^\star$ if $p$ is given by~\eqref{param}.
\begin{proposition}\label{symmetricproblem}
Let $\alpha\in\R$ and assume that $p\in(2,+\infty)$. Then
$$
C_{\alpha,p}^\star =
\begin{cases}
\big(\alpha-\tfrac12\big)^{1+\frac 2p} \, K_{1,p} & \text{if}\ \alpha\geq -\frac12 \,,\\
\big(\alpha+\tfrac32\big)^{1+\frac 2p} \, K_{1,p} & \text{if}\ \alpha\leq -\frac12 \,.
\end{cases}
$$
In particular, we have $C_{\alpha,p}^\star>0$ if and only if $\alpha\not\in\big\{-\frac32,+\frac12\big\}$. In this case, let
$$
\lambda =
\begin{cases}
\big(\alpha - \frac12\big)^2 & \text{if}\ \alpha\geq-\frac12 \,,\\
\big (\alpha+\tfrac32\big)^2 & \text{if}\ \alpha\leq -\frac12 \,.
\end{cases}
$$
Then the optimizers for $C_{\alpha,p}^\star$ are given by
\be{phi:opt}
\phi(s,\omega)=u(s)\,\chi(\omega)\quad\forall\,(s,\omega)\in\R\times\Sp\,,
\ee
where $\chi\in\mathcal H_k$ with $k=0$ if $\alpha>-\frac12$, $k=-\,2$ if $\alpha<-\frac12$ and $k\in\{0,-2\}$ if $\alpha=-\,\frac12$, and where, up to a translation in $s$ and multiplication by a constant, $u$ is given by~\eqref{ustar}. Moreover, $C_{\alpha,p}^\star = \nrmr{u_\star}p^{p-2}$.
\end{proposition}
\noindent Notice that $C^\star_{-\,(\alpha+1),p}=C^\star_{\alpha,p}$ as for the optimal constant for~\eqref{logarithmicSCKN}. However, in contrast to~\eqref{logarithmicSCKN}, the constant $C^\star_{\alpha,p}$ does not vanish for $\alpha\in\Lambda\setminus\big\{-\frac32,+\frac12\big\}$.
So far, we have considered Inequality~\eqref{logarithmicSCKN} in logarithmic coordinates. The results can readily be transferred to Inequality~\eqref{SCKN} in the original coordinates in view of Lem\-ma~\ref{Lem:chvar}. In particular, we obtain the optimal value of $\mathcal C^\star_{\alpha,\beta} = (4\,\pi)^{(p-2)/p}\,C^\star_{\alpha,p}$.

\begin{proof}[Proof of Proposition~\ref{symmetricproblem}] A symmetric spinor $\phi$ can be written in the form~\eqref{phi:opt} with $u\in\mathrm H^1(\R)$ and $\chi\in\mathcal H_k$ for $k\in\{0,-2\}$. Recall that $\mathcal H_0$ is generated by constant spinors and that $\mathcal H_2$ is generated by $\sigma\cdot\omega$ times constant spinors. Since $|(\sigma\cdot\omega)\,\chi|^2=|\chi|^2$, in both cases $|\chi(\omega)|$ is independent of $\omega$. Using $\sigma\cdot L\,\chi = \tilde k\,\chi$, Inequality~\eqref{logarithmicSCKN} becomes
\[
\ir{|u'|^2}+\big(k-\alpha+\tfrac12\big)^2\ir{|u|^2}\ge C^\star_{\alpha,p} \(\ir{|u|^p}\)^{2/p}\,.
\]
In view of the dependence of $\lambda=(k-\alpha+1/2)^2$ on $k$ we deduce that
\begin{itemize}
\item if $\alpha>-\,\frac12$, then minimizing symmetric spinors are necessarily of type $\mathcal H_0$,
\item if $\alpha<-\,\frac12$, then minimizing symmetric spinors are necessarily of type $\mathcal H_2$,
\item if $\alpha =-\,\frac12$, the minimizing symmetric spinors can be of either type $\mathcal H_0$ and $\mathcal H_{-2}$.
\end{itemize}
We deduce that the sharp constant $C_{\alpha,p}^\star$ coincides with $K_{\lambda,p}$ and that the $u$-part of a symmetric optimizer $\phi$ coincides with an optimizer of~\eqref{OneD} found in Lemma~\ref{nagy}. \end{proof}

The minimization problem over a certain class of spinors that is larger than the class of symmetric spinors, but smaller than $\mathcal H_{-2}\oplus\mathcal H_0$, can still be solved explicitly. We~let
\begin{align*}
\mathcal S := \left\{ \psi\in\mathcal D_{\alpha,\beta} :\ \psi(x) = \left( f(|x|) + i\,\sigma\cdot\tfrac x{|x|} \, g(|x|) \right)\chi_0 \ \text{for some functions}\ f,\ g \ \text{on}\ \R_+ \ \right. \\
\left. \text{with}\ \overline f g \ \text{real-valued and some}\ \chi_0\in\C^2 \right\}.
\end{align*}
\begin{proposition}\label{Prop:H-2H0} For all $\alpha\in\R$, $\alpha\leq\beta\leq\alpha+1$ we have
$$
\inf_{\psi\in\mathcal S\setminus\{0\}} \frac{\ird{|\sigma\cdot \nabla \psi (x)|^2\,|x|^{-2\,\alpha}}}{\left(\ird{|\psi(x)|^p\,|x|^{-\beta\,p}}\right)^{2/p}}=\mathcal C_{\alpha,\beta}^\star \,.
$$
The minimizers are of the form 
$$
h(|x|) \left( c_1 + i\,\sigma\cdot\tfrac x{|x|}\,c_2 \right)\chi_0
$$
with $c_1$, $c_2\in\C$, $\chi_0\in\C^2$ and $h(r) = r^{\alpha-1/2} u_\star(\ln(r/R))$ for some $R>0$, where $u_\star$ is given by~\eqref{ustar}, $c_2=0$ if $\alpha>-\frac12$, $c_1=0$ if $\alpha<-\frac12$ and $\overline{c_1}\,c_2\geq 0$ if $\alpha=-\frac12$.
\end{proposition}
\noindent In view of the definition of symmetry given in Section~\ref{Sec:Introduction}, Proposition~\ref{Prop:H-2H0} shows that the case $\alpha=-\,1/2$ does not correspond to symmetry. However, we already know that $C_{-1/2,p}<C_{-1/2,p}^\star$ for any $p\in(2,6)$ by Proposition~\ref{Prop:alonglines} and Theorem~\ref{Thm:symbreak}.

\begin{proof} Let $\psi$ be as in the definition of the class $\mathcal S$ where, without loss of generality, \hbox{$|\chi_0|=1$}. We apply Lemma~\ref{Lem:chvar} with $\phi(s,\omega) = \big(u(s)+ i\,\sigma\cdot\omega\,v(s)\big)\chi_0$, where $u$ and $v$ are related to $f$ and $g$ by $f(r) = r^{\alpha-1/2}\,u(\log r)$ and $g(r) = r^{\alpha-1/2}\,v(\log r)$. Since $\overline uv$ and $\langle\chi_0,\sigma\cdot\omega\,\chi_0\rangle$ are real-valued, we find
$$
|\phi|^2 = |u|^2 + |v|^2
\quad\text{and}\quad
\left|\big(\sigma\cdot L - \alpha+\tfrac12\big)\,\phi\right|^2 = \big(\alpha-\tfrac12\big)^2\,|u|^2 + \big(\alpha+\tfrac32\big)^2\,|v|^2 \,.
$$
Moreover,
$$
|\partial_s\phi|^2 = |u'|^2 + |v'|^2 + 2\,\Re\left( \overline{u'}\,v'\,\langle\chi_0,\sigma\cdot\omega\,\chi_0\rangle \right).
$$
Noting that the integral of the latter term with respect to $\omega\in\Sp$ vanishes, we find that
$$
\frac{\ird{|\sigma\cdot \nabla \psi (x)|^2\,|x|^{-2\,\alpha}}}{\left(\ird{|\psi(x)|^p\,|x|^{-\beta\,p}}\right)^{2/p}}
= \frac{\ic{\(|u'|^2+\big(\alpha-\tfrac12\big)^2\,|u|^2 + |v'|^2 + \big(\alpha+\tfrac32\big)^2\,|v|^2\)}
}{(4\,\pi)^{\frac2p-1}\,\(\ic{\(|u|^2+|v|^2\)^{p/2}}\)^{2/p}} \,.
$$
Our goal is to determine the infimum of this quotient with respect to $u$ and $v$.

Let us assume that $\alpha\geq -\frac12$, the opposite case being similar. Then $\big(\alpha+\tfrac32\big)^2\geq\big(\alpha-\tfrac12\big)^2$. Applying Inequality~\eqref{OneD} and Proposition~\ref{symmetricproblem}, we bound the numerator from below by
\begin{align*}
& \ic{\(|u'|^2+\big(\alpha-\tfrac12\big)^2\,|u|^2 + |v'|^2 + \big(\alpha+\tfrac32\big)^2\,|v|^2\)} \\
& \geq C_{\alpha,p}^\star \(\ic{|u|^p}\)^{2/p} + C_{\alpha,p}^\star \(\ic{|v|^p}\)^{2/p} \\
& \quad + \left( \big(\alpha+\tfrac32\big)^2 - \big(\alpha-\tfrac12\big)^2 \right) \ic{|v|^2} \\
& \geq C_{\alpha,p}^\star\(\ic{(|u|^2+|v|^2)^{p/2}}\)^{2/p} + \left( \big(\alpha+\tfrac32\big)^2 - \big(\alpha-\tfrac12\big)^2 \right) \ic{|v|^2}.
\end{align*}
The last inequality comes from the triangle inequality in $\mathrm L^{p/2}(\R^3,\C)$. We conclude that the infimum over $(u,v)$ coincides with $C_{\alpha,p}^\star$ and, when $\alpha> -\frac12$, it is only attained when $v\equiv 0$, that is, when $\psi$ is symmetric of type $\mathcal H_0$.

To determine the cases of equality for $\alpha=-\,\frac12$, we recall that by the equality conditions in Minkowski's inequality (see, e.g.,~\cite[Theorem~2.4]{MR1817225}) we must have either $u\equiv 0$ or $|v|^2 = \lambda\,|u|^2$ for some $\lambda \geq 0$. From this one easily deduces the result.\end{proof}

As another application of Lemma~\ref{nagy} we now show how it implies Lemma~\ref{Lem:Keller}. This argument is well known and due to Lieb and Thirring, but repeated here for the sake of completeness. By H\"older's inequality with $\gamma$ such that $\gamma+1/2$ and $p/2$ are H\"older conjugates, i.e., such that~\eqref{gamma:p} holds, we notice that
\[
\ir{|u'|^2}-\ir{V\,|u|^2}+\lambda\ir{|u|^2}\ge\nrmr{u'}2^2-\nrmr V{\gamma+1/2}\,\nrmr up^2+\lambda\,\nrmr u2^2 \,.
\]
By Lemma~\ref{nagy}, the right side is nonnegative if $\nrmr V{\gamma+1/2}\le \lambda^{\frac12 + \frac1p}\,K_{1,p}$. An optimization on~$u$ with $\nrmr u2^2=1$ shows that $\lambda_1(V) +\lambda\geq 0$ if $\nrmr V{\gamma+1/2}\le \lambda^{\frac12 + \frac1p}\,K_{1,p}$. As in Lemma~\ref{Lem:Keller}, $-\,\lambda_1(V)$ denotes the lowest eigenvalue of the Schr\"odinger operator $-\,\frac{d^2}{ds^2}-V$. This proves~\eqref{Ineq:Keller} with $c_{\rm LT}(\gamma)=K_{1,p}^{-p/(p-2)}$. Moreover, equality in~\eqref{Ineq:Keller} is achieved if the above inequalities are in fact equalities, that is, if $V$ is proportional to $|u|^{p-2}$ and $u$ is optimal for~\eqref{OneD}, i.e., given up to a translation by~\eqref{ustar}. The proof of Lemma~\ref{Lem:Keller} is completed with $V=V_0$ for $\lambda=4/(p-2)^2$ so that $\gamma^2-1/4=2\,p/(p-2)^2$, and $u_0=u_\star/\nrmr{u_\star}2$ with $u_\star$ given by~\eqref{ustar}.

In connection with Lemmas~\ref{Lem:Keller} and~\ref{nagy} we record the following result on the \emph{P\"oschl-Teller operator}, $-\,\frac{d^2}{ds^2}-\nu\,(\nu+1)\,(\cosh s)^{-2}$, that we find useful. The result is implicitly contained in Lemma~\ref{Lem:Keller}, but we show that it can be obtained directly.
\begin{corollary} If $V(s):=(\cosh s)^{-2}$ for all $s\in\R$, then
\[
\lambda_1\big(\nu\,(\nu+1)\,V\big)=\nu^2\quad\forall\,\nu>0\,.
\]
If $\,V_{A,B}(s):=A\,\big(\cosh(B\,s)\big)^{-2}$ for all $s\in\R$, then
\[
\lambda_1(V_{A,B})=\frac14\(\sqrt{4\,A+B^2}-B\)^2\quad\forall\,(A,B)\in\R^+\times\R^+>0\,.
\]
\end{corollary}
\begin{proof} We read from the proof of Lemma~\ref{nagy} that for any $2<p<\infty$ the function $w(s)=(\cosh s)^{-2/(p-2)}$, i.e., $w(s)=(p/2)^{-1/(p-2)}\,v\big(2\,s/(p-2)\big)$ belongs to $\mathrm H^1(\R)$ and solves
$$
-\,(p-2)^2\,w'' + 4\,w - 2\,p\,(\cosh s)^{-2}\,w = 0 \,.
$$
Thus, $-\,4/(p-2)^2$ is an eigenvalue of the Schr\"odinger operator with potential $-\,2\,p\,(p-2)^{-2}\,(\cosh s)^{-2}$. Since $w$ is positive, it is the smallest. Reparametrizing $\nu\,(\nu+1)=2\,p/(p-2)^2$ yields the first assertion. The second one follows by solving $\nu\,(\nu+1)=A/B^2$ and scaling $s\mapsto B\,s$.
\end{proof}

We refer to~\cite[p.~74]{Landau-Lifschitz-67},~\cite[Section 4.2.2]{MR4496335} and~\cite{poschl1933bemerkungen,MR0121101} for further details on the ground state of the P\"oschl-Teller operator and its spectrum.

\section{Spinor spherical harmonics decomposition}\label{App5}

The purpose of this appendix is twofold. On the one hand, we want to give a self-contained proof of the decomposition of $\mathrm L^2(\Sp,\C^2)$ into eigenspaces of the operator $\sigma\cdot L$ and on the other hand, we want to recall the construction of an explicit basis which will be convenient when doing computations.

We consider the operator
$$
D := \sigma\cdot L +1
\quad\text{in}\quad\mathrm L^2(\Sp,\C^2)\,.
$$
The following result is well-known in physics and in differential geometry. We provide an elementary proof along the lines of~\cite{MR1361548}.
\begin{theorem}
The spectrum of $D$ is discrete and consists of the eigenvalues $\pm (1+\kappa)$, $\kappa\in\N$, with multiplicity $2\,(\kappa+1)$.
\end{theorem}
\begin{proof}
$\rhd$ \emph{Step 1.} Let $\mathcal M_+:=\C^2$, identified with constant spinors in $\mathrm L^2(\Sp,\C^2)$, and
$$
\mathcal M_-:=\big\{ \sigma\cdot\omega\,\Psi\,:\,\Psi\in\C^2\big\}\,.
$$
We claim that
\be{eq:lowest}
\ker(D\mp 1) = \mathcal M_\pm\,.
\ee
Indeed, clearly we have $D\,\psi=\psi$ for $\psi\in\mathcal M_+$ and conversely, if $\psi\in \mathrm H^1(\Sp,\C^2)$ satisfies $D\,\psi =\psi$, then $-\,\Delta\psi =0$ (since $-\,\Delta = (\sigma\cdot L)^2 + \sigma\cdot L$) and consequently $\psi\in\mathcal M_+$. This proves~\eqref{eq:lowest} for the upper signs. The other case in~\eqref{eq:lowest} follows from this using the identity
\be{eq:anticomm}
D\,\sigma\cdot\omega = -\,\sigma\cdot\omega\,D\,,
\ee
see~\eqref{commutation}, which is easily verified. In~\eqref{eq:anticomm}, $\sigma\cdot\omega$ is understood as a multiplication operator.

\medskip\noindent$\rhd$ \emph{Step 2.} We claim that for either choice of the sign, the spectrum of $\big(D\,\pm\tfrac12\big)^2$ consists precisely of the eigenvalues $(\ell+1/2)^2$, $\ell\in\N$, with multiplicities $2\,(2\,\ell+1)$.

To prove this, let $f$ be a $\C$-valued function on $\Sp$. We note that
\begin{equation}
\label{eq:product}
\(D\,\mp\tfrac12\)^2\,(f\,\psi) = \(-\Delta f+ \tfrac14 f\)\psi
\quad\text{if}\quad\psi\in\mathcal M_\pm\,.
\end{equation}
Here $-\,\Delta=L^2$ denotes the Laplace-Beltrami operator on $\Sp$ with the sign convention $-\,\Delta\geq 0$. The proof of~\eqref{eq:product} is immediate for the upper sign, where it follows from $D^2-D=-\,\Delta$ and $\Delta(f\,\psi)=(\Delta f)\,\psi$ since $\psi$ is constant. Identity~\eqref{eq:product} for the lower sign follows from that for the upper sign, since by~\eqref{eq:anticomm} $\big(D+\tfrac12\big)^2\,\sigma\cdot\omega = \sigma\cdot\omega\,\big(D-\tfrac12\big)^2$.

It is well known that the spectrum of the operator $-\,\Delta+\frac14$ acting on scalar functions consists of the eigenvalues $(\ell+1/2)^2$, $\ell\in\N$, with multiplicity $2\,\ell+1$. Therefore, the numbers $(\ell+1/2)^2$, $\ell\in\N$, are eigenvalues of $\big(D\,\pm\tfrac12\big)^2$ of multiplicity $\geq2\,(2\,\ell+1)$. Eigenfunctions are given by functions of the form $Y_{\ell,m}\Psi_i$, where $Y_{\ell,m}$ are eigenfunctions of $-\,\Delta$ and $\{\Psi_1,\Psi_2\}$ is a basis of $\mathcal M_\pm$. Since functions of this form span $\mathrm L^2(\Sp,\C^2)$, we obtain the assertion made at the beginning of this step.

\medskip\noindent$\rhd$ \emph{Step 3.}
We can now complete the proof of the theorem. The computation of the spectrum of $\big(D\,\pm\tfrac12\big)^2$ shows that the spectrum of $D$ is contained in $\Z$ and that one has for $\epsilon\in\{+1,-1\}$
\begin{align}\label{eq:multsum}
\dim\ker\(\big(D+\tfrac\epsilon2\big)-\big(\ell+\tfrac12\big)\) + \dim\ker\(\big(D+\tfrac\epsilon2\big)+\big(\ell+\tfrac12\big)\) & = \dim\ker\(\big(D+\tfrac\epsilon2\big)^2 - \big(\ell+\tfrac12\big)^2\) \notag \\
& = 2\,(2\,\ell+1)\,.
\end{align}
Applying this identity with $\ell=0$ and either choice of $\epsilon$ and recalling~\eqref{eq:lowest}, we deduce that $\dim\ker D=0$. Next, applying~\eqref{eq:multsum} with $\ell=1$ and both choices of $\epsilon$ and recalling~\eqref{eq:lowest}, we deduce that $\dim\ker(D+2)=4=\dim\ker(D-2)$. Next, applying~\eqref{eq:multsum} with $\ell=2$ and both choices of $\epsilon$ and recalling the information about the eigenvalues $\pm 2$, we deduce that $\dim\ker(D+3)=6=\dim\ker(D-3)$. Continuing in this way we obtain the assertion of the theorem.
\end{proof}

\begin{corollary}\label{decomp}
Any function $\psi\in \mathrm L^2(\Sp,\C^2)$ has an orthogonal decomposition
$$
\psi = \sum_{\kappa\in\N,\nu=\pm} \psi_{\kappa,\nu}
$$
with $\|\psi\|^2 = \sum_{k\in\N,\nu=\pm} \|\psi_{k,\nu}\|^2$ such that
$$
D\,\psi_{\kappa,\pm}=\pm\,(1+\kappa)\,\psi_{\kappa,\pm}
\quad\mbox{and}\quad
-\Delta\psi_{\kappa,\nu} = \begin{cases} \kappa\,(\kappa+1)\,\psi_{\kappa,\nu} & \text{if}\quad\nu=+\,, \\
(\kappa+1)\,(\kappa+2)\,\psi_{\kappa,\nu} & \text{if}\quad\nu=-\,.
\end{cases}
$$
Moreover, with $J=L+\tfrac12\,\sigma$ we have
$$
J^2\,\psi_{\kappa,\nu} = \big(\kappa+\tfrac12\big)\,\big(\kappa+\tfrac32\big)\,\psi_{\kappa,\nu}\,.
$$
\end{corollary}

This follows immediately by spectral theory from the previous theorem. For the assertion about the Laplace-Beltrami operator we note that $-\,\Delta = (\sigma\cdot L)^2+\sigma\cdot L$, so
$$
-\Delta = D^2 - D\,,
$$
and $(1+\kappa)^2 - (1+\kappa) = \kappa\,(\kappa+1)$, $(1+\kappa)^2+(1+\kappa) = (\kappa+1)\,(\kappa+2)$. The assertion about $J$ follows from $J^2=L^2+\sigma\cdot L+\tfrac34$ and
$$
\Big( \pm(1+k)\,\big(\pm (1+k) - 1\big) \Big) + \big( \pm (1+k)-1 \big) + \tfrac34 = \big(k+\tfrac12\big)\,\big(k+\tfrac32\big)\,.
$$

Up to now in this appendix, the eigenvalues of $D$ were labelled as $\nu\,(1+\kappa)$ with \hbox{$\kappa\in\N$} and $\nu\in\{+,-\}$. In order to conform with the physics literature we now focus on the eigenvalues of $\sigma\cdot L = D-1$ and label those by $k\in\Z\setminus\{0\}$. Thus, $k=\kappa$ for $\nu=+$ and $k=-\,\kappa-2$ for $\nu=-$, where in both cases $\kappa\in\N$. Then an equivalent way of phrasing Corollary~\ref{decomp} is as follows.

\begin{corollary}\label{decompalt}
There is an orthogonal decomposition
$$
\mathrm L^2(\Sp,\C^2) = \bigoplus_{k\in\Z\setminus\{-1\}} \mathcal H_k
$$
such that, for any $k\in\Z\setminus\{-1\}$, we have $\dim\mathcal H_k = 2\,|k+1|$
and for any $\psi\in\mathcal H_k$ we have
$$
\sigma\cdot L\,\psi = k\,\psi\,,
\quad
L^2\,\psi = k\,(k+1)\,\psi\,,
\quad
J^2\,\psi = \big(k+\tfrac12\big)\,\big(k+\tfrac32\big)\,\psi\,.
$$
\end{corollary}

We now present an explicit basis of the spaces $\mathcal H_k$ following~\cite[pp.~61-62]{MR0096526} or~\cite[p.~127]{MR1219537}. Let $(Y_\ell^m)$, $\ell\in\N$, $m=-\,\ell,\ldots,\ell$, the usual family of spherical harmonics. These are complex-valued functions on $\Sp$ such that, for each $\ell\in\N$, $(Y_\ell^{-\ell},\ldots,Y_\ell^\ell)$ is an orthonormal (with respect to the un-normalized surface measure on $\mathbb S^2$) basis of the eigenspace of the operator $L^2$ corresponding to the eigenvalue $\ell\,(\ell+1)$. The basis can be chosen in such a way that in terms of the components of the operator $L=\omega\wedge(-i\,\nabla)$ we have for each $m=-\,\ell,\ldots,\ell$
\begin{eqnarray}
\label{eq:choicebasis}
L_3\,Y_\ell^m = m\,Y_\ell^m\,,
\quad
(L_1 \pm i L_2)\,Y_\ell^m = \sqrt{(\ell\pm m +1)\,(\ell\mp m)}\,Y_\ell^m\,.
\end{eqnarray}
Since $\overline{L_j f}=-\,L_j\overline f$, we may and will assume that
\begin{equation}
\label{eq:harmonicsconj}
\overline{Y_\ell^m} = (-1)^m\,Y_\ell^{-m}\,.
\end{equation}
In particular, $Y_\ell^0$ is real valued.

We now introduce the functions
$$
\chi_k^m=\frac{\sqrt{4\,\pi}}{\sqrt{2k+1}}\begin{pmatrix}\sqrt{k+m+1/2}\,Y^{m-1/2}_k\\\sqrt{k-m+1/2}\,Y^{m+1/2}_k\end{pmatrix}\,,\quad k \ge 0\,,\quad m=-\,k-1/2, -k+1/2, \dots, k+1/2\,,
$$
and
$$
\chi_k^m= \frac{\sqrt{4\,\pi}}{\sqrt{-2k-1}}\begin{pmatrix}\sqrt{-k-m-1/2}\,Y^{m-1/2}_{-k-1}\\-\sqrt{-k+m-1/2}\,Y^{m+1/2}_{-k-1}\end{pmatrix}\,,\quad k \le -2\,,\quad m= k+3/2, k+1/2, \dots, -k-3/2\,.
$$
(The factor $\sqrt{4\,\pi}$ is due to the fact that we consider the surface measure on $\Sp$ normalized to be a probability measure.)
\begin{lemma}\label{basis}
Let $k\in\Z\setminus\{-1\}$. The functions $\big(\chi_k^{-k-1/2},\ldots,\chi_k^{k+1/2}\big)$ when $k\geq 0$ and the functions $\big(\chi_k^{k+3/2},\ldots,\chi_k^{-k-3/2}\big)$ when $k\leq -2$ form an orthonormal basis of $\mathcal H_k$. Moreover, we have \hbox{$J_3\,\chi_k^m = m\,\chi_k^m$} for every $m$.
\end{lemma}
\begin{proof}
Using~\eqref{eq:choicebasis} one easily verifies that $\sigma\cdot L\,\chi_k^m = k\,\chi_k^m$. Thus, these functions belong to~$\mathcal H_k$. Moreover, again because of~\eqref{eq:choicebasis}, we have $J_3 \chi_k^m = m\,\chi_k^m$. By self-adjointness of $J_3$ we deduce that the $\chi_k^m$ are orthogonal. Since $\dim\mathcal H_k = 2\,|k+1|$ by Corollary~\ref{decompalt}, these functions form a basis of $\mathcal H_k$. Finally, the normalization of $\chi_k^m$ follows immediately from that of the $Y_k^{m\pm 1/2}$.
\end{proof}

\noindent{\bf Acknowledgment:} Partial support through \emph{Conviviality} project (ANR-23-CE40-0003) of the French National Research Agency (J.D.), US National Science Foundation grant DMS-1954995 (R.L.F.) and DMS-2154340 (M.L.), the German Research Foundation grants EXC-2111-390814868 and TRR 352-Project-ID 470903074 (R.L.F.). The authors thank the referees for their comments and questions.\\
\noindent{\scriptsize\copyright\,2026\ by the authors. This paper may be reproduced, in its entirety, for non-commercial purposes. \href{https://creativecommons.org/licenses/by/4.0/legalcode}{CC-BY 4.0}}


\end{document}